\documentclass[12pt]{article}
\usepackage{amsmath, amsthm, amssymb, amscd, mathptmx}

\usepackage{amsfonts}

\usepackage{color}
\usepackage{pstricks}

\usepackage[all]{xy}\CompileMatrices\SelectTips{cm}{12}

\theoremstyle{plain}
\newtheorem{Thm}{\sc Theorem}[section]
\newtheorem{Theorem}[Thm]{\sc Theorem}
\newtheorem{Corollary}[Thm]{\sc Corollary}
\newtheorem{Claim}{\sc Claim}[Thm]
\newtheorem*{Corollary*}{\sc Corollary}

\newtheorem{Proposition}[Thm]{\sc Proposition}
\newtheorem*{Proposition*}{\sc Proposition}
\newtheorem{Lemma}[Thm]{\sc Lemma}

\theoremstyle{definition}
\newtheorem{Definition}[Thm]{Definition}

\theoremstyle{remark}
\newtheorem{Remark}[Thm]{Remark}
\newtheorem{Example}[Thm]{Example}
\newtheorem*{Example*}{Example}
\newtheorem*{Remark*}{Remark}

\renewcommand{\AA}{{\mathbb A}}
\newcommand{\CC}{{\mathbb C}}

\newcommand{\FF}{{\mathbb F}}

\newcommand{\HH}{{\mathbb H}}
\newcommand{\LL}{{\mathbb L}}
\newcommand{\ZZ}{{\mathbb Z}}

\newcommand{\PP}{{\mathbb P}}
\newcommand{\QQ}{{\mathbb Q}}
\newcommand{\RR}{{\mathbb R}}
\newcommand{\VV}{{\mathbb V}}

\newcommand{\cO}{{\mathcal O}}

\newcommand{\HIG}[1]{\mathop{\operatorname{MinHIG}^{#1}}\, }

\newcommand{\Hig}{\mathop{\operatorname{HIG}\, }}
\newcommand{\Mic}{\mathop{\operatorname{MIC}\, }}
\newcommand{\Mod}[1]{\mathop{\operatorname{{#1}-Mod}\, }}
\newcommand{\nilMod}[1]{\mathop{\operatorname{{#1}-Mod}_{\mathop{\rm nil}}\, }}
\newcommand{\ssMod}[2]{\mathop{\operatorname{Min-{#1}-Mod}^{#2}\, }}
\newcommand{\nMod}[2]{\mathop{\operatorname{Min-{#1}-Mod}^{#2}_{\mathop{\rm nil}}\, }}

\newcommand{\PGL}{\mathop{\rm PGL\, }}
\newcommand{\SL}{\mathop{\rm SL\, }}

\newcommand{\et}{{\mathop{\rm et \, }}}
\newcommand{\ch}{{\mathop{\rm ch \, }}}

\newcommand{\Gr}{{\mathop{Gr}}}
\newcommand{\id}{\mathop{\rm Id}}

\newcommand{\im}{\mathop{\rm im \, }}
\newcommand{\Ext}{{\mathop{{\rm Ext \,}}}}

\newcommand{\End}{{\mathop{{\mathcal E}nd}}}

\newcommand{\Tor}{{\mathop{{\mathcal T}or}}}

\newcommand{\Res}{{\mathop{\rm Res}}}

\newcommand{\reduced}{{\mathop{\rm red}}}

\newcommand{\rk}{{\mathop{\rm rk \,}}}

\newcommand{\Sym}{{\mathop{{\rm Sym}}}}

\newcommand{\Supp}{{\mathop{{\rm Supp \,}}}}
\newcommand{\Spec}{{\mathop{{\rm Spec\, }}}}

\begin{document}

\markboth {\rm }{}

\title{Nearby cycles and semipositivity in positive characteristic}
\author{Adrian Langer} \date{\today}

\maketitle


{\noindent \sc Address:}\\
Institute of Mathematics, University of Warsaw,
ul.\ Banacha 2, 02-097 Warszawa, Poland\\
e-mail: {\tt alan@mimuw.edu.pl}

\medskip

\begin{abstract}
  We study restriction of logarithmic Higgs bundles to the boundary
  divisor and we construct the corresponding nearby-cycles functor in
  positive characteristic.  As applications we prove some strong
  semipositivity theorems for analogs of complex polarized variations
  of Hodge structures and their generalizations. This implies, e.g., semipositivity for the
  relative canonical divisor of a semistable reduction in positive
  characteristic and it gives some new strong results generalizing
  semipositivity even for complex varieties.
\end{abstract}

\section*{Introduction}

Let $X$ be a smooth projective variety defined over an algebraically
closed field $k$ of characteristic $p$ and let $D$ be a simple normal
crossing divisor on $X$. In this introduction we assume that $(X, D)$
lifts to the ring $W_2(k)$ of Witt vectors of length at most $2$.

A logarithmic Higgs sheaf on $(X,D)$ is a pair $(E,\theta)$ consisting
of a coherent $\cO_X$-module and an $\cO_X$-linear map $\theta: E\to
E\otimes \Omega_X(\log \, D)$ such that $\theta\wedge
\theta=0$. Equivalently, replacing $\theta$ by $\hat \theta: T_X(\log
\, D) \otimes E\to E$ one can consider a logarithmic Higgs sheaf on
$(X,D)$ as a $\Sym ^{\bullet} T_X(\log \, D)$-module, which is
coherent when considered as an $\cO_X$-module.

Let $\HIG{0} (X,D)$ be the category of locally free logarithmic Higgs
sheaves of rank $r\le p$ on $(X,D)$, which have vanishing Chern
classes in $H^{2*}_{\et}(X, \QQ_l)$ for some $l\ne p$ and are
semistable. In this case semistable means slope $H$-semistable with
respect to some ample divisor $H$, but one can show that the category $\HIG{0}
(X,D)$ does not depend on the choice of $H$. One can also replace
slope semistability by Gieseker semistability and the category remains
the same.

Let $Y$ be an irreducible component of $D$ and let $\imath : Y\to X$ be
the corresponding embedding.
One of the main aims of this paper is to prove the following theorem:

\begin{Theorem}\label{main}
Let $(E,\theta)$ be an object of $\HIG{0} (X,D)$.  The restriction of
$(E, \hat\theta)$ to $Y$ defines a semistable $\Sym ^{\bullet}
\imath^* T_X(\log \, D)$-module.  Moreover, this restriction can be
deformed to an element of $\HIG{0} (Y,D^Y)$, where $D^Y$ is the
restriction of the divisor $D-Y$ to $Y$.
\end{Theorem}

The precise statement of this theorem is contained in Theorem
\ref{nearby-functor} and Corollary
\ref{ss-filtration-of-restriction}. In fact, we prove a more general
version that works also for Higgs sheaves (or modules with an integrable
connection) with non-vanishing Chern classes.

Together with the restriction theorem for curves not contained in the
boundary divisor $D$ (see Theorem
\ref{curve-restriction}) this gives an inductive procedure for
studying restriction of elements of $\HIG{0} (X,D)$ to curves. In
particular, it implies the following theorem (see Definition
\ref{strongly-liftable} for the definition of a strongly liftable
morphism).

\begin{Theorem} \label{semipositivity} 
  Let $(E, \theta)$ be an object of $\HIG{0} (X,D)$.  Let $C$ be a
  smooth projective curve and let $\nu : C\to (X,D)$ be a separable
  morphism that is strongly liftable to $W_2(k)$. Then the induced
  $\Sym ^{\bullet} \nu^* T_X(\log \, D)$-module $\nu ^*E$ is
  semistable.  In particular, if $G$ is a subsheaf of the kernel of
  $\nu^*\theta : \nu^*E\to \nu^* E\otimes \nu^*\Omega_X(\log\, D)$
  then $\deg G\le 0$.
\end{Theorem}

This theorem has an obvious analogue in characteristic zero (see
Theorem \ref{semipositivity-0}). But even the last part of this
theorem was not known in characteristic zero. Already this part
implies essentially all known semipositivity results (see below) for
Higgs bundles or complex polarized variations of Hodge structures due
to Fujita \cite{Fu}, Kawamata \cite{Ka}, Zuo \cite{Zu},
Fujino-Fujisawa \cite[Theorem 5.21]{FF}, Brunebarbe \cite[Theorems 1.8
and 4.5]{Br1}, \cite[Theorem 1.2]{Br2} and many others.  Note that
almost all the proofs of such results are analytic and use Hodge
theory. A notable exception is Arapura's proof of \cite[Theorem 2]{Ar}
that uses reduction to positive characteristic. However, his proof
uses vanishing theorems and it does not give any semipositivity results
in positive characteristic.

We say that a sheaf $E$ on $(X,D)$ is \emph{$W_2$-nef} if for any
smooth projective curve $C$ and any morphism $\nu : C\to (X,D)$ that
is strongly liftable to $W_2(k)$ (see Definition
\ref{strongly-liftable}), we know that all quotients of $\nu^*E$ have
a non-negative degree.

The following corollary is a direct analogue of \cite[Theorem
1.2]{Br2} in positive characteristic. In fact, it implies its
generalization from polystable to the semistable case.

\begin{Corollary} \label{semipositivity-cor} 
 Let $(E, \theta)$ be an object of $\HIG{0} (X,D)$. 
If $E'$ is a locally split subsheaf of $E$ contained in the kernel of $\theta$
then its dual $(E')^*$ is $W_2$-nef on $(X,D)$.
\end{Corollary}

Over complex numbers a typical example of application of such a result
is to semipositivity of direct images of relative canonical sheaves.
This happens also in positive characteristic and we prove the
following result (see Corollary \ref{canonical-Illusie} for a more
precise version).

\begin{Corollary}
  Let $X$ and $Y$ be smooth projective varieties and let $B$ be a
  normal crossing divisor on $Y$. Let $f:X\to Y$ be a smooth
  surjective morphism of relative dimension $d$, which has semi-stable
  reduction along $B$.  Let us set $D=f^{-1}(B)$.  Assume that there
  exists a lifting $\tilde f: (\tilde X, \tilde D)\to (\tilde Y,
  \tilde B)$ of $f$ to $W_2(k)$ with $\tilde f$ a semi-stable
  reduction along $\tilde B$. Assume that $p>d+\dim Y$.  Then
  $R^jf_*(\omega_{X/Y}(D))$ is a $W_2$-nef locally free sheaf on $(Y, B)$ for all integers $j\ge 0$.
\end{Corollary}

This is a positive characteristic analogue of various semipositivity
results due to Griffiths \cite{Gr}, Fujita \cite{Fu}, Kawamata
\cite{Ka}, Fujino--Fujisawa \cite{FF} and others.

In positive characteristic $p$ there are well-known examples due to
L. Moret--Bailly (see \cite[Expos\'e 8]{MB}), who showed for any
integer $n\ge 1$ and any $p$ a family of smooth abelian surfaces $f:
X\to \PP^1$ such that $f_* \omega_{X/\PP ^1}=\cO_{\PP^1} (-n)\oplus
\cO_{\PP^1} (pn)$. In particular, one needs to add some additional
assumptions to be able to get semipositivity results.  The only known
results on semipositivity in positive characteristic concern either
$f_*(\omega_{X/Y}^m(mD))$ for $m\gg 0$ (see \cite{Pa} in case
$\omega_{X/Y}(D)$ is $f$-nef, or \cite{Ej} in case of relative
dimension $1$ or $2$) or they deal with $f_*\omega_{X/Y}$ adding very
strong assumptions on the fibers (see \cite{Szp} for the case
$\dim X=2$ and $\dim Y=1$, and \cite[Theorem 6.4]{Pa2} for a rather
complicated statement).

\medskip

One of the important  results that we prove is the following theorem that is a special case of Theorem
\ref{strong-log-freeness}.

\begin{Theorem}\label{weak-log-freeness}
  Let $E$ be a rank $r$ reflexive sheaf with $c_1(E)=0$ (in $H^{2}_{\et}(X, \QQ _l)$ for some $l\ne
  p$) and $c_2
  (E)H^{n-2}=0$.  Assume that $E$ has a filtration $M_{\bullet}$ such
  that all factors of the filtration are torsion free of rank $\le p$
  with $\mu _{H}(\Gr _j^ME)=\mu _{H}(E)$.  Let us also assume that
  each factor has a structure of a slope $H$-semistable sheaf with an integrable logarithmic
  connection on $(X, D)$.  Then $E$ is locally free and it has
  vanishing Chern classes in $H^{2*}_{\et}(X, \QQ _l)$ for any $l\ne
  p$. Moreover, every quotient $\Gr _j^ME$ is locally free and
  has vanishing Chern classes in $H^{2*}_{\et}(X, \QQ_l)$.
\end{Theorem}

This result can be thought of as an analogue of a graded version of
Schmid's nilpotent orbit theorem (see Remark \ref{Schmid-nilpotent}).
In the case $D=0$  Theorem \ref{weak-log-freeness} gives \cite[Theorem 11]{La2} and fills in a gap in its proof. 
The stronger version, Theorem \ref{log-freeness}, generalizes  Theorem \ref{weak-log-freeness} to the case of Higgs sheaves 
with possibly non-trivial Chern classes and in  the case $D=0$ it is indispensable for the proofs of \cite[Theorem 3.6, 
Corollary-Definition 3.8 and Theorem 3.10]{SYZ}. 
In this last case Theorem \ref{log-freeness} allows to compute higher Chern
classes of twisted preperiodic Higgs bundles. 

\medskip

The structure of the paper is as follows. In Section 1 we recall some
results and prove a few auxiliary results used in the sequel. In
Section 2 we prove Theorem \ref{weak-log-freeness} and we show several
applications.
In Section 3 we construct a nearby-cycles functor and we check that it
preserves some semistability conditions. We also study semistability
of factors of the monodromy filtration associated to residue
endomorphisms of logarithmic Higgs sheaves.  Section 3 is devoted to
applications of these results to semistability and semipositivity of
restriction of Higgs bundles to curves. We also give some geometric
applications to semipositivity of direct images of relative canonical
sheaves. The appendix contains a proof of the functoriality of the inverse
Cartier transform in the logarithmic case.

\subsection*{Notation}

Let $X$ be a smooth variety defined over an algebraically closed field $k$ and let $D$ be a normal crossing divisor on $X$.  We often view $D$ as a closed subscheme of $X$ given locally by one equation but by abuse of notation we also identify $D$ with the corresponding Weil divisor and write $D=0$ instead of $D=\emptyset$. All normal crossing divisors in the paper are reduced simple normal crossing divisors. Sometimes we add ''simple'' to stress the place, 
where we need to use this assumption.

Let us recall that a \emph{logarithmic Higgs sheaf} is a pair
$(E,\theta)$ consisting of a coherent $\cO_X$-module and an
$\cO_X$-linear map $\theta: E\to E\otimes \Omega_X(\log\, D)$ such that
$\theta\wedge \theta=0$.  A \emph{system of logarithmic Hodge sheaves} is a
Higgs sheaf $(E,\theta)$ with a decomposition $E=\bigoplus E^{p,q}$ such that $\theta$ maps $E^{p,q}$ into
$E^{p-1,q+1}\otimes \Omega_X(\log \, D)$.

In this paper if $X$ is projective and we say that a logarithmic Higgs sheaf $(E, \theta)$
is slope $H$-semistable for some ample $H$ then we always implicitly
assume that $E$ is torsion free.   Let us recall that a  system of logarithmic Hodge sheaves
is slope $H$-semistable as a system of logarithmic Hodge sheaves if and only if it is 
slope $H$-semistable as a logarithmic Higgs sheaf (see \cite[Corollary 3.5]{La4}).

\medskip

Now let $S$ be any scheme. We say that $(X,D)$ is a \emph{smooth log pair
  over $S$} if $X$ is a smooth $S$-scheme and $D$ is a relatively
simple normal crossing divisor over $S$. We say that $f: (Y,B)\to (X, D)$ is a morphism
of smooth log pairs if $f: Y\to X$ is a morphism and the support of $B$ contains 
the support of $f^{-1}(D)$.

\medskip
If $E$ is a coherent sheaf of rank $r$ on a smooth projective variety $X$ then we denote by $\ch (E)$
the Chern character of $E$. This is defined as an element of  the rational Chow ring  ${\mathrm{CH}}^* (X)\otimes \QQ$ but
in this paper we abuse notation and denote by $\ch (E)$ the image of this class by the cycle map and we treat it
as an element of the \'etale cohomology ring $H^{*}_{\et}(X, \QQ _l)$, where $l$ is different from the characteristic of the base field (or an element of $H^{*}(X, \QQ )$ in case of complex manifolds). 
By $\Delta (E)$ we denote the discriminant of $E$ defined as $2r c_2(E)-(r-1)c_1^2(E)$. In case of surfaces  we use the degree map $\int_X$ to identify the cohomology group $H^{4}_{\et}(X, \QQ _l)$ (or $H^{4}(X, \QQ)$) with $\QQ _{l}$ (respectively, $\QQ$) and we think of $\Delta (E)$ as an integer. Similarly, in higher dimensions the top degree intersections like $\Delta(E) H^{\dim X-2}$ denote the degree of the cycle $\Delta(E) H^{\dim X-2}$.

\section{Preliminaries}

\subsection{Logarithmic Higgs sheaves}

In this subsection we recall a few results on semistable logarithmic
Higgs sheaves.  Throughout this subsection we fix the following
notation.

Let $X$ be a smooth variety of dimension $n$ defined over an
algebraically closed field $k$ of characteristic $p$.  Let $D$ be a
normal crossing divisor on $X$.

\medskip

Let us recall the following theorem due to Ogus and Vologodsky in the
usual case (see \cite{OV}) and Schepler in the logarithmic one (see
\cite{Sc}; see also \cite[Theorem 2.5]{La3} and \cite[Appendix]{LSYZ}):

\begin{Theorem}\label{log-smooth-OV-correspondence}
  Let us assume that $(X,D)$ is liftable to $W_2(k)$ and let us fix
  such a lifting $(\tilde X, \tilde D)$. There exists a Cartier
  transform $C_{(\tilde X,\tilde D)}$, which defines an equivalence of
  categories of torsion free $\cO_X$-modules with an integrable
  logarithmic connection whose logarithmic $p$-curvature is nilpotent
  of level less or equal to $p-1$ and the residues are nilpotent of
  order less than or equal to $p$, and torsion free logarithmic Higgs
  $\cO_X$-modules with a nilpotent Higgs field of level less or equal
  to $p-1$.
\end{Theorem}

\medskip

From now on in this subsection we assume that $X$ is projective and we
fix an ample divisor $H$ on $X$. Let us recall the following
boundedness result for logarithmic Higgs sheaves.

\begin{Theorem} \label{boundedness}
  Let us fix some number $\Delta$ and a class $c\in H^2_{\et}(X, \QQ_l)$
  for some $l\ne p$. The family of slope $H$-semistable logarithmic
  Higgs sheaves $(E, \theta)$ such that $E$ is reflexive with fixed
  rank $r$, $c_1(E)=c$ and $\Delta(E) H^{n-2}\le \Delta$ is bounded.
\end{Theorem}

\begin{proof} By \cite[Lemma 5]{La2} one can find a constant $C$ such
  that for any rank $r$ slope $H$-semistable logarithmic Higgs sheaves
  $(E,\theta)$ we have $\mu_{\max, H}(E)\le \mu (E)+(r-1)C$.
Hence the result follows from \cite[Theorem 3.4]{La1}.
\end{proof}

Let us note that in the above theorem it is not sufficient to fix $r$,
$c_1(E)H^{n-1}$ and $\Delta(E) H^{n-2}$. We will also need the
following theorem, which is a special case of \cite[Theorem 5.5]{La4}.

\begin{Theorem}\label{deformation-to-system}
  Let $(E, \theta)$ be a slope
  $H$-semistable logarithmic Higgs sheaf. Then there exists a decreasing
  filtration $E=N^0\supset N^1\supset ...\supset N^m=0$ such that
  $\theta (N^i)\subset N^{i-1}\otimes \Omega_X (\log \,  D)$ and the associated
  graded is a slope $H$-semistable system of logarithmic Hodge sheaves.
  \end{Theorem}

\cite[Theorem 5.5]{La4} (see also \cite[Theorem A.4]{LSZ2} in case of flat torsion free sheaves) 
implies also the following result:

\begin{Theorem} \label{existence-of-gr-ss-Griffiths-transverse-filtration}
	If $(E, \nabla)$ is a slope $H$-semistable sheaf with an integrable logarithmic connection 
	then there exists a canonical Griffiths transverse filtration $E=S^0\supset S^1\supset ...\supset S^m=0$ 
	such that the associated graded system of logarithmic Hodge sheaves is slope $H$-semistable.
\end{Theorem}

The canonical filtration $S^{\bullet}$ from Theorem \ref{existence-of-gr-ss-Griffiths-transverse-filtration} is called
\emph{Simpson's filtration} of $(E, \nabla)$. This notion is used in the following generalization of \cite[Theorem
5.12]{La4} and \cite[Theorem 2.2]{LSZ} (see \cite[Theorem 3.1]{La3}).

\begin{Theorem}\label{Higgs-de-Rham}
  Assume the pair $(X,D)$ admits a lifting $(\tilde X, \tilde D)$ to $W_2(k)$. If $(E,
  \theta)$ is a slope $H$-semistable system of logarithmic Hodge
  sheaves of rank $r\le p$ then there exists a canonically defined
  Higgs--de Rham sequence
$$ \xymatrix{
  & (V_0, \nabla _0)\ar[rd]^{\Gr _{S_0}}&& (V_1, \nabla _1)\ar[rd]^{\Gr _{S_1}}&\\
  (E_0, \theta _0)=(E, \theta)\ar[ru]^{C_{(\tilde
  X, \tilde D)}^{-1}}&&(E_1, \theta_1)\ar[ru]^{C_{(\tilde
  X, \tilde D)}^{-1}}&&...\\
}$$ in which each $(V_i, \nabla_i)$ is slope $H$-semistable and
$(E_{i+1}, \theta _{i+1})$ is the slope $H$-semistable system of
logarithmic Hodge sheaves associated to $(V_i, \nabla _i)$ via its
Simpson's filtration $S^{\bullet}_i$.
\end{Theorem}

\medskip

The following theorem is a generalization of \cite[Theorem 10]{La2} to
the logarithmic case.  We skip its proof as it is the same as in the
non-logarithmic case.

\begin{Theorem} \label{Bogomolov-restriction-for-Higgs} 
	Assume the pair $(X,D)$ is  liftable to $W_2(k)$.
	Let $d_0$ be a non-negative integer such that $T_X(-\log \,
	D)\otimes \cO_X(d_0H)$ is globally generated. Let $(E,  \theta)$ be 
	a slope $H$-stable logarithmic Higgs sheaf of rank $r\le p$.
	Let us take an integer $$d>\frac{r-1}{ r}\Delta (E)H^{n-2} +\frac{1}{r(r-1)H^n} .$$
	Moreover, if $r>2$ let us also assume that $d>2(r-1)^2d_0$. 
	Let $Y\in |dH|$ be a smooth divisor such that $E_Y$ has
	no torsion and $D\cap Y$ is a normal crossing divisor on $Y$. Then the
	logarithmic Higgs sheaf $(E_Y, \theta_Y)$ induced from  
	$(E,\theta)$ via restricting to $Y$ and composition $E_Y\to E_Y\otimes \Omega_X(\log \, D)|_Y\to
	E_Y\otimes \Omega_Y (\log \, D\cap Y)$, is slope $H_Y$-stable.
\end{Theorem}

\begin{Corollary} \label{Bogomolov-restriction-for-Higgs-ss} 
	Assume the pair $(X,D)$ is  liftable to $W_2(k)$ and let $d_0$ be as in the previous theorem. Let $(E,  \theta)$ be 
	a slope $H$-semistable logarithmic Higgs sheaf of rank $r\le p$ and let $d$ be an integer satisfying the same 
	inequalites as in the previous theorem.
	Then for a general divisor   $Y\in |dH|$  the restriction  $(E_Y, \theta_Y)$ is  slope $H_Y$-semistable.
\end{Corollary}

\begin{proof}
Let $M_{\bullet}$ be a Jordan--H\"older filtration of $(E, \theta)$. By definition this means that 
all the quotients $\Gr _i^ME$ are slope $H$-stable logarithmic Higgs sheaves with slopes $\mu_H (\Gr _i^ME)$
equal to $\mu _H (E)$. Existence of such a filtration for logarithmic slope $H$-semistable logarithmic Higgs sheaves
is standard and follows by the same arguments as for the usual slope $H$-stable sheaves
(see, e.g., \cite[1.5 and 1.6]{HL}). To simplify notation let us set $E_i=\Gr _i^ME$ (we consider it as a logarithmic Higgs sheaf and not only a sheaf) and $r_i=\rk E_i$. 
Then the Hodge index theorem implies that 
$$
\frac{\Delta (E)H^{n-2}}{ r}= \sum \frac{\Delta
	(E_i)H^{n-2}}{ r_i} -\frac{1}{r}\sum_{i<j} r_ir_j \left(
\frac{c_1E_i}{r_i}-\frac{c_1E_j}{ r_j}\right) ^2
H^{n-2}
\ge \sum \frac{\Delta (E_i)H^{n-2}}{ r_i}.
$$
Therefore our assumptions on $d$ imply that we can apply Theorem \ref{Bogomolov-restriction-for-Higgs}
to each quotient $E_i$. So if we choose a smooth divisor $Y\in |mH|$ such that $D\cap Y$ is a normal crossing divisor on $Y$ and $(E_i)_Y$ has no torsion for every $i$ then  the restricted logarithmic Higgs sheaf $(E_i)_Y$ is slope $H$-stable and
hence $(E_Y, \theta_Y)$ is  slope $H_Y$-semistable. 
Note that general $Y\in |mH|$ satisfies the above assumptions by Bertini's theorem and Lemma \ref{HL-Corollary}.
\end{proof}

\begin{Remark} \label{Bogomolov-restriction-for-connections-remark}
	In Theorem \ref{Bogomolov-restriction-for-Higgs} and Corollary \ref{Bogomolov-restriction-for-Higgs-ss}
we can replace a logarithmic Higgs sheaf with a sheaf with an integrable logarithmic connection. The proofs of the results remain the same.
\end{Remark}

\medskip

Let us also recall Bogomolov's inequality for logarithmic Higgs
sheaves (see \cite[Theorem 3.3]{La3} for a more general version).

\begin{Theorem} \label{log-Bogomolov-with-lifting} 
	Assume that $(X,D)$ admits a lifting to $W_2(k)$.  Then for any
	slope $H$-semistable logarithmic Higgs sheaf $(E, \theta)$ of rank
	$r\le p$ we have
	$$\Delta (E) H^{n-2}\ge 0.$$
\end{Theorem}

\begin{Remark} \label{Bogomolov-inequality-for-connections}
The above theorem holds also for sheaves with an integrable logarithmic connection.
Indeed, if $(E, \nabla)$ is a rank $r\le p$ slope $H$-semistable sheaf with an integrable 
logarithmic connection and $S^{\bullet}$ is its  Simpson's filtration then 
by the above theorem
$$\Delta (E) H^{n-2}=\Delta (\Gr _S E) H^{n-2}\ge 0.$$
\end{Remark}

\subsection{Higher discriminants}

Let us fix  a smooth projective variety $X$ defined over an arbitrary
algebraically closed field $k$.

Let $E$ be a rank $r>0$ coherent sheaf on $X$. Let us fix a prime $l$
non-equal to the characteristic of the base field $k$ and let us write
$$\log (\ch (E))=\log r+\sum _{i\ge 1} (-1)^{i+1}\frac{1}{i!r^i}\Delta_i(E)$$
for some classes $\Delta_i(E)\in H^{2i}_{\et}(X, \QQ _l)$ that we call
\emph{higher discriminants} of $E$ (we can also use $\Delta_i(E)\in
H^{*}(X, \QQ )$ in case of complex manifolds). These discriminants are
polynomials in Chern classes of $E$ with integral coefficients. They
are variants of Drezet's logarithmic invariants (with somewhat
different normalization to get integral coefficients and
$\Delta_2(E)=\Delta(E)$). Note that for any line bundle $L$ we have
$\Delta_i(E\otimes L)=\Delta_i(E)$ for $i\ge 2$.  This follows
immediately from the fact that
$$\log (\ch (E\otimes L))= \log (\ch (E) \cdot \ch(L)) =\log (\ch (E))+c_1(L).$$
In the following we often use this property of discriminants without further notice.

\begin{Lemma} \label{equivalence-Delta}
  The following conditions in $H^{*}_{\et}(X, \QQ _l)$ (or in
  $H^{*}(X, \QQ )$ in case of complex manifolds) are equivalent:
\begin{enumerate}
\item $r^ic_i(E)=\binom{r}{i}c_1(E)^i$ for all $i\ge 1$,
\item $\Delta_i(E)=0$ for all $i\ge 2$,
\item $\log \ch (E)=\log r+\frac {c_1(E)}{r}$.
\end{enumerate}
\end{Lemma}

\begin{proof}
Equivalence of 2 and 3 is clear as $\Delta_1(E)=c_1(E)$.
For simplicity of notation let us assume that $E$ is locally free. Proof in the 
general case is the same except that we need to replace $E$ by its class in $K^0(X)$
and do all the computations in Grothendieck's $K$-group.

By the Bloch--Gieseker covering trick (see \cite[Lemma 2.1]{BG})
there exists a smooth projective
variety $\tilde X$ and a finite flat surjective covering $f: \tilde
X\to X$ together with a line bundle $L$ such that $f^*(\det
E)^{-1}=L^{\otimes r}$. Let us set $\tilde E:= f^*E \otimes L.$ 
Note that $c_1(\tilde E)=0$ and
$$\Delta _i (\tilde E)=\Delta _i(f^* E)$$
for all $i\ge 2$.

Since  $f$ induces an injection $H^{*}_{\et}(X, \QQ _l)\to H^{*}_{\et}(\tilde X,
\QQ _l)$, the second condition is equivalent to the vanishing of $\Delta _i (\tilde E)$ for all $i\ge 1$,
i.e., to the equality $\log \ch (\tilde E)=\log r$. Clearly, this is equivalent to 
$c_i(\tilde E)=0$ for all $i\ge 1$.

For all $i\ge 0$ we have
$$c_i(\tilde E)=c_i(f^*E\otimes L)=\sum_{j=0}^{i}\binom{r-j}{i-j} c_1(L)^{i-j}c_j(f^*E).$$ 
Using $c_1(L)=-\frac {1}{r}c_1(f^*E)$ and the fact that the map
$H^{*}_{\et}(X, \QQ _l)\to H^{*}_{\et}(\tilde X, \QQ _l)$ is injective, we see that the second condition is equivalent to
the equalities
\begin{equation}\label{equ}
\sum_{j=0}^{i}\binom{r-j}{i-j} (-c_1(E))^{i-j}r^jc_i(E)=0
\end{equation}
for all $i=1,...,r$.
This follows from the fact that the equalities (\ref{equ}) for $i\le m$ 
are equivalent to the equalities 
\begin{equation}\label{equ2}
r^ic_i(E)=\binom{r}{i}c_1(E)^i
\end{equation}
for $i=1,...,m$. We prove this by induction on $m$. For $m=1$ it is clear, so let us assume 
it holds for $1,..., m-1$. We can assume that (\ref{equ2}) holds for $i<m$.
Then we have
\begin{eqnarray*}
\sum_{i=0}^{m}\binom{r-i}{m-i} (-c_1(E))^{m-i}r^ic_i(E)=r^mc_m(E)-\binom{r}{m}c_1(E)^m\\
+
\sum_{i=0}^{m}(-1)^{m-i} \binom{r-i}{m-i} \binom{r}{i} c_1(E)^m=
r^mc_m(E)-\binom{r}{m}c_1(E)^m\\
+ \binom{r}{m} c_1(E)^m\cdot \sum_{i=0}^{m}(-1)^{m-i} \binom{m}{i}=r^mc_m(E)-\binom{r}{m}c_1(E)^m.
\end{eqnarray*}
This proves that under our assumptions,
(\ref{equ}) for $i=m$ is equivalent to (\ref{equ2}) for $i=m$.
\end{proof}

\subsection{Criterion for local freeness and restriction to divisors} \label{criterion}

Let $X$ be an integral noetherian scheme and let $E$ be a coherent sheaf of $\cO_X$-modules.
Let $S(E)$ be the set of points $x\in X$ such that $E_x$ is not a free $\cO_{X,x}$-module. We call $S(E)$
the \emph{singular set} of $E$.

Let us define the function $\varphi: X\to \ZZ$ by $\varphi (x)= \dim _{k(x)}(E\otimes k(x))$.
Let $\eta$ be the generic point of $X$. For a point $x\in X$,
by \cite[Chapter II, Lemma 8.9]{Ha}, $E_x$ is a free $\cO_{X,x}$-module if and only if $\varphi (x)=\varphi (\eta)$. 
On the other hand, by Nakayama's lemma the function $\varphi$ is upper semicontinuous 
(see \cite[Chapter III, Example 12.7.2]{Ha}), so $S(E)=\{ x\in X: \varphi (x)> \varphi (\eta) \}$ is closed.

In the following we say that $E$ \emph{is locally free outside a finite number of points} if $S(E)$ 
is a finite set of points. 

\medskip

Now let $X$ be a smooth projective variety  of dimension $n$
defined over an arbitrary algebraically closed field $k$.  
In the following we will use several times the following criterion for local freeness of graded 
sheaves associated to filtrations.

\begin{Lemma}\label{loc-free-passing-to-graded}
  Let us assume that $n\ge 3$ and let $V$ be a reflexive sheaf on
  $X$ with a filtration $N^m=0\subset N^{m-1}\subset ...\subset N^0=V$ such that each $N^i$ is saturated in $V$.
Let $W=\bigoplus N^i/N^{i-1}$ be  the associated graded and let us assume that 
\begin{enumerate}
\item the reflexivization $W^{**}$ of $W$ is locally free, and
\item $W$ is locally free outside a finite number of points. 
\end{enumerate}
Then both $V$ and $W$ are locally free.
\end{Lemma}

\begin{proof}
  It is sufficient to prove the lemma for $m=2$. The general case
  follows easily from this special one by induction on the length $m$ of the
  filtration.

Assuming $m=2$ our assumptions imply that $N^1$ is locally free and we have a short exact sequence
$$0\to N^0/N^1\to (N^0/N^1)^{**}\to T\to 0$$
for some sheaf $T$ supported on a finite number of points. By
assumption we also know that $(N^0/N^1)^{**}$ is locally free.  Let us
note that by Serre's duality $\Ext ^2(T, N^1)$ is dual to $\Ext ^{n-2}
(N^1, T\otimes \omega_X)= H^{n-2}(T\otimes \omega_X\otimes (N^1)^*)=0$
as $n\ge 3$.  Hence by the long $\Ext$ exact sequence, the canonical
map $\Ext ^1((N^0/N^1)^{**}, N^1)\to \Ext ^1(N^0/N^1, N^1)$ is
surjective. Therefore there exists a coherent sheaf $\tilde V$ such that the
following diagram is commutative:
$$\xymatrix{
  0\ar[r]&N^1\ar@{=}[d]\ar[r]& V\ar[d]\ar[r] &N^0/N^1 \ar[d]\ar[r]& 0 \\
  0\ar[r]&N^1\ar[r]& \tilde V\ar[r] & (N^0/N^1)^{**} \ar[r]& 0. \\
}$$ But since $V$ is reflexive and $\tilde V$ is locally free, the map
$V\to \tilde V$ is an isomorphism (as it is an isomorphism outside of
the support of $T$).  Hence $T=0$ and $W^{**}=W$. This immediately implies 
the required assertion.
\end{proof}

\medskip 

We will also need the following lemmas allowing us to keep track of singularities of sheaves when restricting to divisors.

\begin{Lemma}\label{HL-Corollary}
Let $\Lambda$ be a base point free linear system on  $X$  and let $E$ be a coherent  $\cO_X$-module.
\begin{enumerate}
\item If $E$ is reflexive and $Y\in \Lambda$ is integral then $E_Y$ is a torsion free $\cO_Y$-module.
\item If  $E$ is torsion free (reflexive) and $Y\in \Lambda$ is   general 
 then $E_Y$ is also torsion free (reflexive, respectively) as an $\cO_Y$-module.
\end{enumerate}
\end{Lemma}

The above lemma follows from \cite[Lemma 1.1.12 and Corollary 1.1.14]{HL}.

\begin{Lemma}\label{loc-free-0} Let $E$ be a rank $r$ torsion free sheaf on $X$ and let $Y$ be an integral divisor on $X$
	such that $E_Y$ is locally free. Then $S(E)\cap Y=\emptyset$, i.e., $E$ is locally free at all points of $Y$.
	Moreover, if $Y$ is ample then $E$ is locally free outside a finite number of points.
\end{Lemma}

\begin{proof}
	Since every torsion free sheaf on a smooth curve is locally free, we can assume that the dimension $n$ of $X$ is greater than $1$. Since $S(E)$ has codimension $\ge 2$ in $X$, there exists a codimension $1$ point $y\in Y-S(E)$. 
	Let $\eta$ be the generic point of $X$ and $\eta'$ the generic point of $Y$. Since $E_Y$ is locally free, $E_{Y,y}$ is a free $\cO_{Y,y}$-module and hence	
	$$\dim _{k(y)} E_{Y, y}\otimes _{\cO_{Y, y}}k(y) = \dim _{k(\eta ')} E_{Y, y}\otimes _{\cO_{Y, y}}k(\eta ')=\rk E_Y.$$
	By the choice of $y$ the $\cO_{X, y}$-module  $E_{y}$ is  free 	and hence 
	$$\dim _{k(y)} E_y\otimes _{\cO_{X, y}}k(y) = \dim _{k(\eta)} E_y \otimes _{\cO_{X, y}}k(\eta )=r.$$
	Since $E_y\otimes _{\cO_{X, y}}k(y)\simeq E_{Y, y}\otimes _{\cO_{Y, y}}k(y)$ we see that  $E_Y$ has rank $r$. 
	By assumption for any point $z\in Y$ the $\cO_{Y,z}$-module  $E_{Y,z}$ is free, so we get
	$$\dim _{k(z)} E_z\otimes _{\cO_{X, z}}k(z) =\dim _{k(z)} E_{Y, z} \otimes _{\cO_{Y, z}}k(z) = 
	\dim _{k(\eta ')} E_{Y, z}\otimes _{\cO_{Y, z}}k(\eta ')= r.$$
	Then \cite[Chapter II, Lemma 8.9]{Ha} implies that $E_z$ is a free $\cO_{X, z}$-module, which proves the first assertion.
	
	Now let us assume that $Y$ is ample. 
	The singular set $S(E)$ is a closed subset of $X$ and  $S(E)\cap Y=\emptyset$, so it does not have any 
	irreducible components of dimension $\ge 1$. So  $S(E)$ is zero-dimensional.
\end{proof}

\section{Local freeness}

In this section we fix the following notation.  Let $X$ be a smooth
projective variety of dimension $n\ge 2$ defined over an algebraically
closed field $k$ of characteristic $p$ and let $D$ be a simple normal
crossing divisor on $X$.  We assume that $D\subset X$ admits a lifting
to $W_2(k)$.  We also fix an ample divisor $H$ on $X$.

\medskip

The main aim of this section is to prove the following generalization of Theorem \ref{weak-log-freeness}:

\begin{Theorem}\label{strong-log-freeness}
  Let $E$ be a rank $r$ reflexive sheaf with $\Delta (E)H^{n-2}=0$.
  Assume that $E$ has a filtration $M_{\bullet}$ such that all factors
  of the filtration are torsion free of rank $\le p$ with $\mu
  _{H}(\Gr _j^ME)=\mu _{H}(E)$.  Let us also assume that each factor
  has a structure of a slope $H$-semistable logarithmic Higgs sheaf
  (or of a slope $H$-semistable sheaf with an integrable logarithmic connection) on $(X, D)$.
  Then $E$ is locally free and
  $$c_m(E)=\binom{r}{m} \left( \frac{c_1(E)}{r}\right) ^m$$ 
  in $H^{2m}_{\et}(X, \QQ _l)$ for all $m\ge 1$ and any $l\ne
  p$. Moreover, every quotient $\Gr _j^ME$ is locally free and for all
  $m\ge 1$ we have
 $$c_m(\Gr _j^ME)=\binom{r_j}{m}\left( \frac{c_1(E)}{r}\right) ^m$$ 
in $H^{2m}_{\et}(X, \QQ_l)$, where  $r_j=\rk \Gr _j^ME$.
\end{Theorem}

This theorem is a strong version of the following theorem to which we will reduce its proof.

\begin{Theorem} \label{log-freeness}
  Let $(E, \theta)$  ($(E, \nabla )$) be a rank $r\le p$ slope $H$-semistable
  logarithmic Higgs sheaf (a rank $r\le p$ slope $H$-semistable sheaf with a logarithmic connection, respectively). 
  Then the following conditions are equivalent:
  \begin{enumerate}
  	\item $\Delta (E)H^{n-2}=0$ and $E$ is reflexive,
  	\item $\Delta (E)H^{n-2}=0$ and $E$ is locally free, 
  	\item $c_m(E)=\binom{r}{m} \left( \frac{c_1(E)}{r}\right) ^m$ 
  	in $H^{2m}_{\et}(X, \QQ _l)$ for all $m\ge 1$ and any $l\ne
  	p$.  
  \end{enumerate} 
\end{Theorem}

\begin{Remark}
To simplify notation in Theorems \ref{strong-log-freeness} and \ref{log-freeness} we deal with only one polarization
although one can also replace $H$ by a collection of ample divisors as in, e.g., \cite[Theorem 10]{La2}.
\end{Remark}

Theorem \ref{log-freeness} generalizes \cite[Theorem 11]{La2} to the case of
logarithmic Higgs sheaves with possibly non-trivial Chern classes.  It
also generalizes \cite[Theorems 3.6 and  3.10]{SYZ}, which deal with systems of
Hodge sheaves of rank $r<p$ on $X$ defined over $k=\bar \FF_p$. In
this last case Theorem \ref{log-freeness} allows to compute higher Chern
classes of twisted preperiodic Higgs bundles.  Let us also remark that
a special case of the above result was implicitly used in proof of
\cite[Theorem 3]{Ar} (see Remark \ref{Arapura}).

The strategy of our proof of Theorem \ref{log-freeness}  in the case $D=0$ is modelled on 
the proof of \cite[Theorem 11]{La2}. Unfortunately, the proof of \cite[Theorem 11]{La2} contains a serious
gap: it is not clear that the family of Higgs sheaves
$\{(E_i,\theta_i)\}$ considered in the proof is bounded as a priori
the sheaves $E_i$ need not be reflexive.  However, if one assumes that
in \cite[Theorem 11]{La2} all Chern classes vanish, then the arguments
there show that $E$ is locally free. This is already sufficient for almost all the
applications mentioned in \cite{La2} (except for Corollary 6 that also
needs an additional assumption on vanishing Chern classes; one also needs to
slightly adjust the proof of \cite[Corollary 5]{La2}).

In general, one can easily find examples of Higgs--de Rham sequences
starting with a locally free sheaf for which other sheaves in the
sequence are not reflexive. This causes several complications that we
need to overcome. The same error appeared independently in the first version
of \cite[Theorem 3.10]{SYZ}, where the authors claimed existence of a
certain map on the open subset of the moduli space of semistable
sheaves, parameterizing reflexive sheaves. However, in case of
\cite[Theorem 3.10]{SYZ}, it is not so easy to adjust the arguments
adding additional assumptions (this would require at least Lemma
\ref{equivalence-Delta} and repeating the proof of \cite[Theorem
11]{La2}). So Theorem \ref{log-freeness} offers in this case the only 
available proof.

A new idea appearing in the general  proof of Theorem \ref{log-freeness}, when compared 
to the case $D=0$,  is that we need to use a nearby cycles functor to prove local freeness of the restriction of
$E$ to the irreducible components of $D$.

\subsection{Reduction from filtrations to sheaves} \label{reduction}

In this subsection we show how to reduce the proof of Theorem \ref{strong-log-freeness}
to Theorem \ref{log-freeness}. Before we do that let us prove a few  independent lemmas:

\begin{Lemma}\label{JH-factors}
	Let $E$ be a rank $r$ reflexive sheaf with $\Delta (E)H^{n-2}=0$.
	Assume that $E$ has a filtration $M_{\bullet}$ such that all factors
	of the filtration are torsion free of rank $\le p$ with $\mu
	_{H}(\Gr _j^ME)=\mu _{H}(E)$.  Let us also assume that each factor
	has a structure of a slope $H$-semistable logarithmic Higgs sheaf
	(or of a slope $H$-semistable integrable logarithmic connection) on $(X, D)$.
	Then for all $j$ we have $\Delta (\Gr _j^ME) H^{n-2}=0$ and $ c_1(\Gr _j^ME)/ \rk (\Gr _j^ME)= c_1(E)/r$ 
	in $H^{2}_{\et}(X, \QQ _l)$ for any $l\ne p$.
\end{Lemma}

\begin{proof}
	To simplify notation let us set $E_i=\Gr _i^ME$ and $r_i=\rk E_i$. 
	Since $\mu _{H}(E_i)=\mu _{H}(E_j)$,  the Hodge index theorem and Theorem \ref{log-Bogomolov-with-lifting} imply that
	$$\aligned
	0=\frac{\Delta (E)H^{n-2}}{ r}&= \sum \frac{\Delta
		( E_i)H^{n-2}}{ r_i} -\frac{1}{r}\sum_{i<j} r_ir_j \left(
	\frac{c_1 (E_i)}{r_i}-\frac{c_1 (E_j)}{ r_j}\right) ^2
	H^{n-2}\\
	&\ge \sum \frac{\Delta ( E_i)H^{n-2}}{ r_i}\ge 0.\\
	\endaligned$$ Hence ${\Delta (E_i)} H^{n-2}=0$ and 
	$\left(
	\frac{c_1 (E_i)}{r_i}-\frac{c_1 (E_j)}{ r_j}\right) ^2
	H^{n-2}=0$ for all $i$ and $j$. By assumption we also have
	$\left(
	\frac{c_1 (E_i)}{r_i}-\frac{c_1 (E_j)}{ r_j}\right) 
	H^{n-1} = \mu
	_{H}(E_i) - \mu
	_{H}(E_j)=0.$
	
	Now let us recall that by \cite[Theorem 9.6.3]{Kl} if $B$ is a divisor on $X$ such that
	$BH^{n-1}=B^2H^{n-2}=0$, then the class of $B$ in the group of divisors on $X$ modulo algebraic equivalence is torsion.
	 Using the cycle map we obtain equality $B=0$ in $H^{2}_{\et}(X, \QQ _l)$ for any $l\ne p$.
	
	Applying this fact to $B=\left( \frac{c_1 (E_i)}{r_i}-\frac{c_1 (E_j)}{ r_j}\right) $
	we get the required equalities.  
\end{proof}

\medskip

\begin{Lemma}\label{codim-3}
	Let $(E, \theta)$ ( $(E, \nabla)$) be a
	rank $r\le p$ slope $H$-semistable logarithmic Higgs sheaf (sheaf with an integrable logarithmic connection, respectively) with
	$\Delta (E)H^{n-2}=0$.  Then $\theta$ ($\nabla$) extends uniquely to a logarithmic Higgs field 
	$\tilde \theta $ (an integrable logarithmic connection $\tilde \nabla$) on the reflexivization $E^{**}$ so that $(E^{**},
	\tilde \theta) $  ($(E^{**}, \tilde \nabla) $, respectively) is slope $H$-semistable.  
		Moreover, $\Delta (E^{**})H^{n-2}=0$
	and the canonical map $E\to E^{**}$ is an isomorphism outside of a
	closed subset of codimension $\ge 3$.
\end{Lemma}

\begin{proof}
	Equality $\theta \wedge \theta =0$ implies $\theta ^{**} \wedge \theta ^{**} =0$, so  $\tilde \theta:= \theta ^{**}$ is a logarithmic Higgs field. In the second case  we can extend $\nabla$, e.g., in the following way. Let us set $U:= X-S(E)$ and let $j: U\hookrightarrow X$ be the corresponding embedding. Then $E^{**}=j_*(j^*E)$ and we can define 
	$\tilde \nabla$  by $\tilde \nabla = j_* \nabla j^*$. It is easy to see that this is a logarithmic connection. It is integrable, because $\tilde \nabla \wedge \tilde \nabla : E^{**}\to E^{**}\otimes \Omega_X^2 (\log D)$ is an $\cO_X$-linear map extending $\nabla _U\wedge \nabla_ U =0$.

	Note that for any torsion free sheaf $G$ the line bundles $\det (G^{**})$ and $\det (G)$ are isomorphic on $X-S(G)$ and $S(G)$ has codimension $\ge 2$. So $\det (G^{**})\simeq \det (G)$ and $c_1(G^{**})=c_1(G)$. 
	Now for any subsheaf $G\subset E^{**}$ we have $(E\cap G)^{**}=G^{**}$ as both sheaves are reflexive and equal outside of codimension $\ge 2$. So the sheaf $E$ contains subsheaf $E\cap G$ of the same slope as $G$.
	This shows that passing to the reflexivization preserves slope $H$-semistability (and also slope $H$-stability). 		  
	
	To prove the second part note
	that the canonical map $E\to E^{**}$ is injective as by assumption
	$E$ is torsion free.  Let $T$ be the cokernel of this map. Without
	any loss of generality we can assume that $H$ is very ample. After
	restricting to a general complete intersection surface $Y\in
	|H|\cap...\cap |H|$, we get a short exact sequence
	$$0\to E_Y\to (E^{**})_Y\to T_Y\to 0.$$
	Since $T$ is supported in codimension $\ge 2$, $T_Y$ is supported on a
	finite number of points. We have
	$$0=\Delta(E)H^{n-2}=\Delta (E_Y)= \Delta ((E^{**})_Y)+h^0(Y, T_Y) =\Delta(E^{**})H^{n-2}+h^0(Y, T_Y).$$
	Since $(E^{**}, \theta ^{**})$ is slope $H$-semistable, by
	Bogomolov's inequality for logarithmic Higgs bundles (see Theorem
	\ref{log-Bogomolov-with-lifting} and Remark \ref{Bogomolov-inequality-for-connections}) we have $\Delta(E^{**})H^{n-2}\ge
	0$. Hence we get $\Delta(E^{**})H^{n-2}=0$ and $h^0(Y, T_Y)=0$. Since
	$T_Y$ is supported on a finite number of points, we get $T_Y=0$.  It
	follows that $T$ is supported in codimension $\ge 3$.
\end{proof}

\medskip

\begin{Lemma}\label{lifting}
	Replacing $H$ by some its multiple we can assume that any $Y\in |H|$ is liftable to 
	$\tilde Y\subset \tilde X$. Moreover, for any closed point  $x\in U:=X-\Supp D$ 
	a general divisor $Y\in |H|$ passing through $x$ is smooth and the divisor $D+Y$ is a normal crossing divisor.
	Then  $D_Y=D\cap Y$ is a normal crossing divisor on $Y$ and the pair $(Y, D_Y)$ is liftable to $W_2(k)$.  
\end{Lemma}

\begin{proof}
	The first part follows from the proof of \cite[Theorem 11]{La2}. Replacing $H$ by its multiple we can also assume that $H$ is very ample. As in the proof of \cite[Theorem 3.1]{DH} we can also assume that the subsystem $\Lambda\subset |H|$ consisting of all divisors containing $x$ has $x$ as its scheme-theoretic base locus. If $\pi: X'\to X$ is the blow up of $x$ then, replacing $H$ if necessary by its multiple, we can also assume that $\pi^*H-E$ is very ample (see \cite[Chapter II, Proposition 7.10]{Ha}).
Let $\{D_i\}_{i\in I}$ be the irreducible components of $D$  viewed as reduced closed subschemes of $X$. Let us set $D_i'=\pi^{-1} (D_i)$ and $D'=\pi^{-1} (D)$. By Bertini's theorem for any $J\subset I$, general $Y'\in |\pi^*H-E|$ intersects all irreducible components of $\bigcap_{j\in J}D_j'$ along smooth divisors. Then  $D'+Y'$ is a normal crossing divisor on $X'$. By \cite[Theorem 2.1]{DH} the image of a general divisor $Y'\in |\pi^*H-E|$ is smooth and it is a general divisor in $\Lambda$.
Hence for general $Y\in \Lambda$, $D+Y$ is a normal crossing divisor on $X$. Moreover, $\tilde D+\tilde Y\subset \tilde X$ is its lifting to $W_2(k)$. This implies that also $(\tilde Y, \tilde Y\cap \tilde D)$ lifts $(Y, D_Y)$ to $W_2(k)$.
\end{proof}

\medskip

\begin{Lemma}\label{implication}
	Theorem \ref{log-freeness} in dimension $\le n$ implies Theorem \ref{strong-log-freeness}
	in dimension $\le n$.
\end{Lemma}

\begin{proof}
	The proof is by induction on the dimension $n$ of $X$. If $n=1$ then
	the assertion follows from the fact that torsion free sheaves on a
	smooth curve are locally free. Assume that the implication holds for
	varieties of dimension less than $n$ and let $X$ be of dimension $n$.
	
	First we consider the case in which  each factor of the filtration from Theorem \ref{strong-log-freeness}
	 has a structure of a slope $H$-semistable logarithmic Higgs sheaf. Let us write $E_j$ for $\Gr _j^ME$ and $r_j$ for its rank.   Replacing $H$ by its multiple we can assume that $T_X(-\log \, D)\otimes \cO_X(d_0H)$ is globally generated. Moreover, by Lemma \ref{lifting} we can assume that a general divisor $Y\in |H|$  the pair $(Y, D_Y=D\cap	 Y)$ is log smooth and  liftable to $W_2(k)$. 
	 By Corollary \ref{Bogomolov-restriction-for-Higgs-ss}  applied to each	quotient of the filtration $M_{\bullet}$, for large $d$ and for a general section $Y\in |dH|$, the restriction of each quotient $E_j:=\Gr _j^ME$ to $Y$ is a
	slope $H_Y$-semistable logarithmic Higgs sheaf and  the restriction $E_Y$ is reflexive (here we use
	Lemma \ref{HL-Corollary}).  
	Hence by the induction assumption each $(E_j)_Y$ is locally free.  So by Lemma \ref{loc-free-0}
	each  $E_j$ is locally free outside a finite number of points of $X$.
	
	Since by Lemma \ref{JH-factors} we have ${\Delta (E_j)} H^{n-2}=0$,  Theorem \ref{log-freeness} applied to $X$ 
	implies that all $E_j^{**}$ are locally free. Hence the assumptions of Lemma
	\ref{loc-free-passing-to-graded} are satisfied and we conclude that
	$E$ and  all quotients $E_j$ are locally free.  By Theorem \ref{log-freeness}
	this implies that
	$$c_m(E_j)=\binom{r_j}{m}\left( \frac{c_1(E _j)}{r_j}\right) ^m,$$ 
	which with equality  $c_1(E_j)=\frac{r_j}{r}c_1(E)$  finishes the proof
	of the second part of  Theorem \ref{strong-log-freeness}. Now a simple computation of Chern classes shows that
	we also have 
	$$c_m(E)=\binom{r}{m} \left( \frac{c_1(E)}{r}\right) ^m.$$ 	

\medskip
	
		Now let us consider the case in which  each factor of the filtration $M_{\bullet}$ from Theorem \ref{strong-log-freeness} has a structure of a slope $H$-semistable sheaf with an integrable logarithmic connection.
		The same arguments as above allow us to prove that for general $Y$ as above the restriction 
		$(\Gr ^M E)_Y$ is locally free  (here we use Remark \ref{Bogomolov-restriction-for-connections-remark} instead of  Corollary \ref{Bogomolov-restriction-for-Higgs-ss}). So	$\Gr ^M E $ is locally free  outside a finite number of points and  by Lemma \ref{loc-free-passing-to-graded} it is sufficient to prove that $(\Gr ^M  E)^{**}$ is locally free.
Then $\Gr ^M E $ is locally free and we can finish as in the case of logarithmic Higgs sheaves.

Let us set  $E_j=\Gr _j^ME$ and  $r_j=\rk E_j$.  For general $Y$  the restriction $(M_i)_Y$ is a subsheaf of $E_Y$,
so $(M_{\bullet}) _Y$ is a filtration of $E_Y$. 
 If $n>2$ then  as above we have $\Delta ((E_j)_Y) H^{n-3}=0$ and hence by the induction assumption applied to $E_Y$ we have	for all $m\ge 1$
		  $$c_m((E_j)_Y)=\binom{r_j}{m}\left( \frac{c_1(E_Y)}{r}\right) ^m.$$
For $n=2$ such equalities are clear as we need to check them only for $m=1$.
		
Let us recall that  by Lemma \ref{JH-factors} we have  $\Delta (E_j) H^{n-2}=0$. 
So by Lemma \ref{codim-3} $\bar E_j:=(E_j)^{**}$ is a slope $H$-semistable sheaf with an integrable logarithmic connection and we have $\Delta (\bar E _j) H^{n-2}=0$.	Theorem \ref{existence-of-gr-ss-Griffiths-transverse-filtration} allows us to construct a filtration $S^{j, \bullet}$ of  $\bar E _j$ such that the associated graded $\Gr _{S^j} \bar E_j $ is a slope $H$-semistable Higgs sheaf with $\Delta (\Gr _{S^j} \bar E _j) H^{n-2}=0$.  Again using Lemma \ref{codim-3}, we see that 
 $(\Gr _{S^j} \bar  E _j)^{**}$ satisfies condition (1) of  Theorem \ref{strong-log-freeness} and hence it is locally free.

Note that for general $Y$ as above we have $(S^i_j)_Y\subset (\bar E_j)_Y$ and  $(\Gr _{S_j} \bar E_j)_Y= \Gr _{(S_j)_Y} (\bar E_j)_Y$. Since  $E_j$ is locally free along $Y$ we have $ (\bar E_j)_Y=(E_j)_Y$. Therefore the Chern classes of   $(\Gr _{S_j} \bar E_j)_Y$ satisfy condition (3) of Theorem \ref{log-freeness} and thus $(\Gr _{S_j} \bar E_j)_Y$ is locally free. 
So by Lemma \ref{loc-free-0} the sheaf $\Gr _{S_j} \bar E_j$ is locally free outside a finite number of points. Now we can use Lemma \ref{loc-free-passing-to-graded} to conclude that $\bar E _j$ is locally free. This proves that 
$(\Gr ^M  E)^{**}=\bigoplus _j \bar E_j$ is locally free as required.

\end{proof}

\subsection{Local freeness for  sheaves}

In this subsection we show the proof of Theorem \ref{log-freeness}.

\medskip

\medskip

\begin{proof}[Proof of Theorem \ref{log-freeness}]

We prove the required assertion by induction on the dimension $n$ of
$X$. Let us assume that $n=2$. Then equivalence of (1) and (2) is obvious since every reflexive sheaf 
on a smooth surface is locally free. The fact that (2) implies (3) is also obvious as equality 
in (3) for $m=1$ is trivial and for $m=2$ it is equivalent to $\Delta (E)=0$.
The fact that (3) implies (1) follows from Lemma \ref{codim-3}.

\medskip

Now let us assume that $n\ge 3$ and equivalence of conditions (1), (2) and (3) holds for 
varieties of dimension less than $n$.  Replacing $H$ by its multiple we can assume $T_X(-\log \, D)\otimes \cO_X(d_0H)$ is globally generated and  by Lemma \ref{lifting} we can also assume that a general divisor $Y\in |H|$  the pair $(Y, D_Y=D\cap	 Y)$ is log smooth and  liftable to $W_2(k)$. 

First let us prove that (1) implies (2) and (3). 
Let $(E,\theta)$ ($(E, \nabla)$) be a reflexive rank $r\le p$ slope $H$-semistable
logarithmic Higgs sheaf (sheaf with an integrable logarithmic connection) with $\Delta (E)H^{n-2}=0$.

\begin{Claim} \label{easy-vanishing}
	We have $\Delta_i(E)=0$ for $2\le i<n$.
\end{Claim}

\begin{proof} For large $d$ and 
	for a general hyperplane section $Y\in |dH|$,  by Corollary \ref{Bogomolov-restriction-for-Higgs-ss} 
	(or  Remark \ref{Bogomolov-restriction-for-connections-remark})
	we know that $(E_Y, \theta_Y)$  ($(E_Y, \nabla_Y)$, respectively) is slope $H_Y$-semistable. By
	Lemma \ref{HL-Corollary}
	we also know that the restriction $E_Y$ is reflexive. 
	Since $\Delta (E_Y)H_Y^{n-3}=d\cdot \Delta (E)H^{n-2}=0$, by the induction assumption $E_Y$ is locally free 
	and $c_m(E_Y)=\binom{r}{m} \left( \frac{c_1(E_Y)}{r}\right) ^m$  in $H^{2m}_{\et}(Y, \QQ _l)$ for all $m\ge 1$ and any $l\ne p$.  
	By Lemma \ref{equivalence-Delta} this implies equalities $\Delta_i(E_Y)=0$ for $2\le i<n$. 
	By the Lefschetz hyperplane theorem, the inclusion $Y\hookrightarrow X$ induces injections
	$H^{2i}_{\et}(X, \QQ _l)\to H^{2i}_{\et}(Y, \QQ _l)$ for $i<n$, which proves the claim.
\end{proof}

\begin{Claim} \label{semistable-loc-free}
	If $(E, \theta)$  (or $(E, \nabla)$ ) is slope $H$-stable then $E$ is locally free.
\end{Claim}

\begin{proof}
  Since  $E$ is reflexive for any smooth hypersurface $Y\in |H|$ the restriction $E_Y$ is torsion free (see Lemma \ref{HL-Corollary}).  Then, possibly replacing $H$ with some its multiple, Lemma \ref{lifting} implies that for every closed point $x\in U:=X-\Supp D$  we can find  $Y\in |H|$ passing through $x$ such that the pair $(Y, D_Y=D\cap	 Y)$ is log smooth and  liftable to $W_2(k)$.  We can use  Theorem  \ref{Bogomolov-restriction-for-Higgs} 	
  (or Remark \ref{Bogomolov-restriction-for-connections-remark}) to conclude that 
  for any such $Y$ the restriction $(E_Y, \theta_Y)$  ($(E_Y, \nabla_Y)$, respectively)
  is slope $H_Y$-stable. Since  $\Delta_i(E)=0$ for $2\le i\le n$ we get $\Delta_i(E_Y)=0$ for $2\le i\le \dim
  Y=n-1$. So $E_Y$ satisfies (3) and our induction assumption implies 
  that $E_Y$ is locally free. Then Lemma \ref{loc-free-0} 
  implies that $E$ is locally free at all points of $Y$. 
  This shows that $E$ is locally free on $U$. 
  
  Now let $Y$ be an irreducible component of $D$. To finish the proof  it is sufficient to show that $E_Y$
  is locally free as then $E$ is locally free along $Y$ by Lemma \ref{loc-free-0}. This part requires the results of Subsections 3.1-3.4 (that do not depend on Theorems \ref{strong-log-freeness} and \ref{log-freeness}).
  The proof is similar but more complicated than that of Theorem \ref{nearby-functor}.  As far as possible we will keep the notation from that proof and show  the necessary adjustments.
 
 We construct a certain sequence analogous to the canonical Higgs--de Rham sequence of $(E, \theta)$ ($(E, \nabla)$) 
 in the following way.
  By Theorem \ref{deformation-to-system} there exists a decreasing Griffiths transverse filtration  
  $N^{\bullet }$ of $E$ such that the  associated graded  $(\bar E_0,\bar \theta_0):= \Gr _N (E, \theta)$  is 
  a slope $H$-semistable system of logarithmic Hodge sheaves (in particular, $\bar \theta_0$ is nilpotent). 
  In case of logarithmic connections we use Simpson's filtration $S^{\bullet }$ instead of $N^{\bullet }$.
  Then  
 using Lemma \ref{codim-3} we define  $(E_0,\theta_0)$ as $((\bar E_0)^{**}, {\bar \theta_0}^{**} )$.  Lemma \ref{codim-3}
 implies that $\Delta (E_0) H^{n-2}=0$ and  $(E_0,\theta_0)$ is slope $H$-semistable.  Now we define 
 $(V_{0}, \nabla _0):= C_{(\tilde X,\tilde D)}^{-1} (E_0, \theta _0) $.  
 Let $S_0^{\bullet}$  be (decreasing) Simpson's filtration on $(V _0, \nabla _0) $ and let 
  $(\bar E_1=\Gr_{S_0}(V _0),\bar \theta _1 ) $ be the associated system of Hodge sheaves. Then we set $(E_1,\theta_1):=((\bar E_1)^{**},
  {\bar \theta_1}^{**} )$ and repeat the procedure. In this way we get the following sequence
  	$${  \xymatrix{
  (E, \theta) \ar[rd]&	&& (V_0, \nabla _0)\ar[rd]^{\Gr _{S_0}}&&& (V_1, \nabla _1)\ar[rd]^{\Gr _{S_1}}&\\
  &	(\bar E_0, \bar \theta _0)\ar[r]&	(E_0, \theta _0)\ar[ru]^{C^{-1}}&&	(\bar E_1, \bar \theta _1)\ar[r]&(E_1, \theta_1)\ar[ru]^{C^{-1}}&&...\\
  }}$$
in which  each logarithmic Higgs sheaf  $(E_j, \theta_j)$ is reflexive rank $r\le p$  slope $H$-semistable with  $\Delta (E_j) H^{n-2}=0$. This follows by induction as 
 $\Delta (\bar E_j)H^{n-2}= \Delta (V_j) H^{n-2}= p^2 \Delta (E_{j-1})H^{n-2}=0$ and then Lemma \ref{codim-3} gives 
 $\Delta (E_j) H^{n-2}=0$. In case of logarithmic connections the sequence is the same except that we replace $(E, \theta)$
 by $(E, \nabla)$.
An easy induction shows also that $c_1(E_j)=p^j c_1(E)$ for all $j\ge 0$.

  Now let us write $p^m=rs_m+q_m$ for some non-negative integers $s_m$
  and $0\le q_m<r$. Let us set $(G_m, \theta _{G_m}):=(E_m, 
  \theta _m)\otimes \det E^{-s_m}$.  Then  $\Delta _i(G_m)=\Delta _i(E_m)=0$ for $2\le i <n$ 
 and $c_1(G_m)= q_m c_1(E)$ can take only  finitely many values. So Theorem
  \ref{boundedness} implies that the family of reflexive slope
  $H$-semistable logarithmic Higgs sheaves $\{(G_m, \theta _{G_m})\} _{m\ge 0}$ is
  bounded. It follows that the family of sheaves  $\{(G_m)_Y\} _{m\ge 0}$ is
  also bounded.

  	Let $E'_0$  be an $\LL _Y^0$-submodule of the $\LL_Y^0$-module $((E_0)_Y, \theta _0 |_Y)$.
  	Note that $E_m$ is locally free outside a finite number of points and 	by Lemma \ref{codim-3}  $\bar E_m$ is isomorphic to $E_m$ outside of a closed subset of codimension $\ge 3$. So all $\bar E_m$ are locally free outside of a closed subset of codimension $\ge 3$. In particular, $(\bar E_m)_Y$ is locally free outside of a closed subset of codimension $\ge 2$. 
  	This, similarly as in proof of Theorem \ref{nearby-functor}, allows us to construct an $\LL _Y^0$-submodule $E_1'\subset (( E_1 )_Y, \theta_1 |_Y )$ such that 
  	$\mu_{H_Y}(E'_1)=p \mu_{H_Y}(E')$. More precisely, 	as in the proof of Theorem \ref{nearby-functor}  $E_0'$ induces an $\LL _Y$-submodule $V'_0 $ of $((V_0)_Y, (\nabla_0) |_Y)$.
  	We have a filtration  $\bar S^{\bullet}_{Y}$ of
  	$(V_0)_Y$ defined by $\bar S^j_{Y}:=\im ((S^j_0)_Y\to (V_0)_Y)$. Note that 	
  	$$(\Gr_{\bar S_Y} ((V_0)_Y))^{**}= ((\Gr_{S_0} V_0)_Y)^{**}= ((\bar E_1)_Y)^{**}=((E_1)_Y)^{**}$$
  	as all sheaves are reflexive and isomorphic on the set where $(\bar E_1)_Y$ is locally free, i.e., outside of a closed subset of codimension $\ge 2$ in $Y$. Now $V'_0\subset (V_0)_Y$ has a filtration induced from $\bar S^{\bullet}_{Y}$
 and the reflexivization of the associated graded is a subsheaf of $((E_1)_Y)^{**}$ that after intersecting with $(E_1) _Y$
 gives the required submodule.

 Repeating the above procedure allows us to construct
  	a sequence $\{E'_m\} _{m\ge 0}$ of $\LL_Y^0$-modules such that
  	$E'_m\subset ((E_m)_Y, \theta_m|_Y)$ and $\mu_{H_Y}(E'_m)=p^m
  	\mu_{H_Y}(E')$. Then as in the proof of Theorem \ref{nearby-functor} the  
boundedness of the family $\{(G_m)_Y\} _{m\ge 0}$ implies that the $\LL_Y^0$-module $((E_0)_Y, \theta _0 |_Y)$ is semistable.
But we know that  $\Delta_i((E_0)_Y)=\Delta_i(E_Y)=0$ for $i\ge 2$, so Lemma \ref{equivalence-Delta} and our induction assumption show that $(E_0)_Y$ is locally free. This implies that $E_Y$ is locally free as required.
\end{proof}

Now we can prove that $E$ is always locally free.  Let $M_{\bullet}$ be 
a Jordan--H\"older filtration of $(E, \theta)$ (or $(E, \nabla )$) and let us set 
$E_i=\Gr_i^M (E)$ and $r_i=\rk E_i$. Then by Lemma \ref{JH-factors} we know that 
${\Delta (E_i)} H^{n-2}=0$ and for all $i$ we have $c_1(E_i)=\frac{r_i}{r}c_1(E)$ in
$H^{2}_{\et}(X, \QQ _l)$ for $l\ne p$. 
By Theorem \ref{Bogomolov-restriction-for-Higgs} (Remark \ref{Bogomolov-restriction-for-connections-remark}, respectively)
for large $d$ and a general 
smooth hypersurface $Y\in |dH|$ the restriction $(E_i)_Y$ is a slope $H_Y$-stable 
logarithmic sheaf. Since $E$ is reflexive  by Lemma \ref{HL-Corollary} the restriction $E_Y$ is also reflexive 
for general $Y$.
Therefore by the induction assumption $E_Y$ is locally free. Moreover, our induction assumption and Lemma \ref{implication} 
imply that all the factors $(E_i)_Y$ are also locally free. So by Lemma \ref{loc-free-0} all $E_i$ are locally free outside a finite number of points.

However, we also know that the logarithmic Higgs sheaf (respectively, the sheaf with an integrable connection) $E_i^{**}$ is
slope $H$-stable and ${\Delta (E_i^{**})} H^{n-2}=0$. So by Claim \ref{semistable-loc-free} 
all sheaves $E_i^{**}$ are locally free. Hence we can apply Lemma
\ref{loc-free-passing-to-graded} to conclude that $E$ is locally free. This finishes the proof that (1) 
implies (2).

\medskip

To finish  the proof that (1) implies (3) note that by Theorem \ref{deformation-to-system} there
exists a decreasing filtration $E=N^0\supset N^1\supset ...\supset
N^m=0$ such that $\theta (N^i)\subset N^{i-1}\otimes \Omega_X (\log \,  D)$
(in case of logarithmic connections we use Simpson's filtration) and the
associated graded system $(E_0, \theta _0)$ of logarithmic Hodge
sheaves is slope $H$-semistable. Let us recall that by Claim \ref{easy-vanishing} we already know that 
$\Delta_i(E)=0$ for $2\le i<n$. Hence $\Delta_i(E_0)=0$ for $2\le i<n$ and  if we take large $d$ and a general divisor $Y\in |dH|$ then by Corollary \ref{Bogomolov-restriction-for-Higgs-ss} the restriction $(E_0)_Y$ satisfies (3) on $Y$.
So by the induction assumption $(E_0)_Y$ is locally free, which by Lemma \ref{loc-free-0} implies that $E_0$ is
locally free outside a finite number of points.  Since we already know that (1) implies (2), 
we see that $E_0^{**}$ is locally free. Then Lemma \ref{loc-free-passing-to-graded} implies that $E_0$ is locally free.

Now let us consider the canonical Higgs-de Rham sequence starting with $(E_0,\theta_0)$
(see Theorem \ref{Higgs-de-Rham})
$$ \xymatrix{
  & (V_0, \nabla _0)\ar[rd]^{\Gr _{S_0}}&& (V_1, \nabla _1)\ar[rd]^{\Gr _{S_1}}&\\
  (E_0, \theta _0)\ar[ru]^{C^{-1}}&&(E_1, \theta_1)\ar[ru]^{C^{-1}}&&...\\
}$$ 
where for simplicity we write $C^{-1}$ to denote the inverse Cartier transform.
By definition each $(V_m, \nabla_m)$ is slope $H$-semistable and
each $(E_{m+1}, \theta _{m+1})$ is the slope $H$-semistable
logarithmic system of Hodge sheaves associated to $(V_i, \nabla _i)$ via Simpson's
filtration.

\begin{Claim}\label{stable-Higgs-de-Rham-sequence}
All $E_m$ are locally free.
\end{Claim}

\begin{proof}
  Each $(E_m^{**}, \tilde \theta_m)$ is slope $H$-semistable and since (1) implies (2) 
  it is also locally free. Note also that
  $\Delta_i(E_m)=p^{im}\Delta (E)=0$ for $i<n$, so the same argument as in the case of $E_0$ shows that
  each sheaf $E_m$ is locally free outside a finite number of points.  Now we prove by induction on $m$
  that $E_m$ and $V_m$ are locally free. For $m=0$ we already know
  that $E_0=E$ is locally free and hence $V_0=C^{-1}(E_0)$ is also
  locally free.  So let us assume that $V_{m-1}$ is locally free. Then
  Lemma \ref{loc-free-passing-to-graded} implies that $E_m$ is locally
  free and hence also $V_m=C^{-1}(E_m)$ is locally free, which
  finishes the induction.
\end{proof}

Now let us write $p^m=rs_m+q_m$ for some non-negative integers $s_m$
and $0\le q_m<r$. Let us set $(G_m, \theta _{G_m}):=(E_m, \tilde
\theta _m)\otimes \det E^{-s_m}$. By \cite[Lemma 2]{La2} we have
$\Delta_i (G_m)=\Delta_i(E_m)=p^{im}\Delta _i(E)$ for $1\le i\le n$,
so $\Delta (G_m)=0$.  Note also that  $c_1(G_m)= q_m c_1(E)$ can take only 
finitely many values,  so Theorem
\ref{boundedness} implies that the family of locally free slope
$H$-semistable logarithmic Higgs sheaves $\{(G_m, \theta _{G_m})\} _{m\ge 0}$ is
bounded. In particular, the set $\{\Delta_n (G_m)\}_{m\ge
	0}=\{p^{nm}\Delta _n(E)\}_{m\ge 0}$ is finite. Hence $\Delta
_n(E)=0$, which finishes the proof of vanishing of $\Delta _i(E)$ for all $2\le i\le n$.
Now Lemma \ref{equivalence-Delta} implies that for all $m\ge 1$ 
$$c_m(E)=\frac{\binom{r}{m}}{r^m} c_1(E)^m$$ 
in $H^{2m}_{\et}(X, \QQ _l)$. This finishes the proof that (1) implies (3).

\medskip 
Clearly (2) implies (1), so it is sufficient to prove that (3) implies (1). 
Let us consider a rank $r\le p$ slope $H$-semistable logarithmic Higgs sheaf  $(E, \theta)$  such that $c_m(E)=\frac{\binom{r}{m}}{r^m} c_1(E)^m$ for all $m\ge 2$. 
By Lemma \ref{codim-3} we have $\Delta (E^{**})H^{n-2}=0$. Since (1) implies (3) we know that 
$(E^{**}, \tilde \theta)$ satisfies $c_m(E^{**})=\frac{\binom{r}{m}}{r^m} c_1(E^{**})^m$ for all $m\ge 1$.
As in the proof of Lemma \ref{codim-3} we see that $c_1(E^{**})=c_1(E)$.  So our assumptions imply that 
$c_m(E^{**})=c_m (E)$ for all $m\ge 1$. Since $E$ and $E^{**}$ have the same rank,
the Riemann--Roch theorem implies that the Hilbert polynomials of $E$ and $E^{**}$ are equal.  Let $T=E^{**}/E$. 
Then the short exact sequence
$$0\to E(m)\to E^{**} (m)\to T(m)\to 0$$
shows that  the Hilbert polynomial of $T$ 
is trivial. So $T=0$ and  $E$ is reflexive. In case of a sheaf with an integrable logarithmic connection
the proof of implication $(3)\Rightarrow (1)$ is exactly the same.
\end{proof}

\medskip

Lemma \ref{implication} and Theorem \ref{log-freeness}  immediately imply the following corollary:

\begin{Corollary} \label{log-freeness2} 
	Let $(E, \theta)$ be a rank $r\le p$ slope $H$-semistable logarithmic Higgs sheaf. 
	Let us assume that $E$ is reflexive and $\Delta (E)H^{n-2}=0$.
  If $(G, \theta _G)$ is a rank $s$ factor in a slope
 $H$-Jordan-H\"older filtration of $(E, \theta)$ then it is locally
 free and for all $m\ge 1$ we have
 $$c_m(G)=\binom{s}{m} \left( \frac{c_1(E)}{r}  \right)^m$$ 
 in $H^{2m}_{\et}(X, \QQ _l)$ for $l\ne p$.
\end{Corollary}

The following corollary is a direct generalization \cite[Theorem
11]{La2} to the logarithmic case.

\begin{Corollary} \label{loc-free-theorem} 
 Let $(E, \theta)$ be a rank $r\le p$ slope $H$-semistable logarithmic Higgs
  sheaf with $\ch _1(E)H^{n-1}=0$ and $\ch _2(E)H^{n-2}=0$. Assume
  that either $E$ is reflexive or the normalized Hilbert polynomial of
  $E$ is the same as that of $\cO_X$. Then $(E,\theta)$ has a
  filtration whose quotients are locally free slope $H$-stable
  logarithmic Higgs sheaves with vanishing Chern classes.
\end{Corollary}

\begin{proof}
  By Theorem \ref{log-Bogomolov-with-lifting} we have
  $\Delta(E)H^{n-2}\ge 0$.  So by the Hodge index theorem we get
$$0=2r \, {\ch } _2 (E)H^{n-2}=c_1(E)^2H^{n-2}- \Delta (E)H^{n-2}\le c_1(E)^2H^{n-2} 
\le \frac{(c_1(E)H^{n-1})^2}{H^n}=0.$$ Hence we have
$\Delta(E)H^{n-2}=0$ and $c_1(E)^2H^{n-2}=0$. Since $c_1(E)H^{n-1}=0$
this implies that $c_1(E)=0$ (see proof of Lemma \ref{JH-factors}). If $E$ is reflexive then the
corollary follows directly from Corollary \ref{log-freeness2}.
In the second case we argue as in the proof  that (3) implies (1) in Theorem \ref{log-freeness}.
Namely, $E^{**}$ satisfies condition (1) of Theorem \ref{log-freeness}
and hence $c_m (E^{**})=0$ for all $m\ge 1$. Then the Hilbert polynomials of $E$ and $E^{**}$ are equal.
So the Hilbert polynomial of $T=E^{**}/E$ is trivial. This implies that $T=0$ and $E$ is reflexive, which 
reduces us to the previous case.
\end{proof}

\medskip

\begin{Remark}
  A special case of the implication $(3)\Rightarrow (2)$ in Theorem \ref{log-freeness} was proven 
  in \cite[Proposition 3.12]{LSZ} using Faltings's result on Fontaine
  modules.
\end{Remark}

\begin{Remark}
	Theorem \ref{log-freeness} implies that all the sheaves $E_i$ and $V_i$
	appearing in the canonical  Higgs-de Rham sequence of a system of logarithmic Hodge sheaves, 
	which satisfies the equivalent conditions of Theorem \ref{log-freeness} and has a
	nilpotent Higgs field, are locally free. This follows also from the proof of Theorem \ref{log-freeness}
	(see Claim \ref{stable-Higgs-de-Rham-sequence}).
\end{Remark}

\begin{Remark} \label{Arapura} In proof of \cite[Lemma 4.4]{Ar} and
  \cite[Lemma 4.5]{Ar} (needed for \cite[Theorem 3]{Ar}) the author implicitly uses that $B(E,\theta)$
  is locally free if $(E,\theta)$ is locally free. More precisely, he
  applies \cite[Lemma 4.3]{Ar} to $B(E,\theta)$ and this fails if
  $B(E,\theta)$ is not locally free. It is easy to find examples for
  which $(E,\theta)$ is semistable, $E$ is locally free but
  $B(E,\theta)$ is not even reflexive. In particular, in both
  \cite[Lemma 4.4]{Ar} and \cite[Lemma 4.5]{Ar} one needs to assume
  that $(E,\theta)$ is semistable with vanishing Chern classes and
  then use our Theorem \ref{log-freeness} (see the above remark).
  
  Note also that at the time of writing \cite{Ar}, Theorem \ref{log-freeness}
  was not claimed in the logarithmic case that was used there. In the logarithmic case, 
  even if $k=\bar \FF_p$ and one has vanishing of all Chern classes, the method 
  of proof of local freeness from \cite[Proposition 3.12]{LSZ} does not apply.
\end{Remark}

\subsection{Local freeness in characteristic zero}

By a standard spreading-out argument Theorem \ref{strong-log-freeness}, Theorem \ref{log-freeness} and Corollary \ref{loc-free-theorem} imply the following generalization of \cite[Theorem 2]{Si} to the
logarithmic case.

\begin{Theorem}\label{free-log-Simpson}
  Let $X$ be a smooth projective variety defined over a field of
  characteristic zero and let $D$ be a normal crossing divisor
  on $X$.  Let $H$ be an ample divisor on $X$ and let $(E, \theta)$ be
  a slope $H$-semistable logarithmic Higgs sheaf with $\ch
  _1(E)H^{n-1}=0$ and $\ch _2(E)H^{n-2}=0$. Then the following conditions are equivalent:
   \begin{enumerate}
  	\item  $E$ is reflexive,
  	\item  $E$ is locally free, 
  	\item the normalized Hilbert polynomial of $E$ is the same as
  	that of $\cO_X$,
  	\item $E$ has vanishing rational Chern classes, i.e., $c_m(E)=0$ in $H^{2m} (X, \QQ )$ for all $m\ge 1$,
  	\item $(E, \theta )$ has a filtration whose quotients are locally free slope $H$-stable logarithmic Higgs sheaves with 
  	vanishing rational Chern classes.
  \end{enumerate} 
\end{Theorem}

\begin{Remark}
The same theorems show that in the above theorem we can replace a logarithmic Higgs sheaf by a sheaf with an integrable logarithmic connection. 
\end{Remark}

\begin{Remark} \label{Schmid-nilpotent} Let $(V, \nabla)$ be a
  polarized variation of Hodge structures on $X-D$ with unipotent
  monodromy along the irreducible components of $D$. Let $(\tilde V,
  \tilde  \nabla)$ be Deligne's canonical extension of $(V, \nabla)$ with
  nilpotent residues along the irreducible components of $D$. Then Schmid's
  nilpotent orbit theorem implies that the Hodge filtration on $V$
  extends to a filtration of $\tilde V$ with locally free subquotients.

  Note that it is easy to see that $\tilde V$ has vanishing Chern classes in the de Rham cohomology of $X$
  (as all residues are nilpotent) and hence it also has vanishing Chern classes in $H^{2*} (X, \QQ)$. 
  Similarly, all the subobjects of $(\tilde V, \tilde  \nabla)$ have vanishing rational Chern classes. In
  particular, $(\tilde V, \tilde  \nabla)$ is slope semistable. Therefore
  Theorem \ref{existence-of-gr-ss-Griffiths-transverse-filtration}
  gives Simpson's filtration such that the associated graded $(E, \theta )$ is slope semistable. 
 Since $E$ has  vanishing rational Chern classes, Theorem \ref{free-log-Simpson} implies that $E$ is locally
  free. Note that \cite[Corollary 5.6]{La4} implies that the associated
  graded of Simpson's filtration of  $(\tilde V, \tilde  \nabla)$ coincides with the associated graded of
  the filtration obtained by Schmid's theorem. Moreover, if the
  associated graded $(E, \theta )$  is slope stable then the corresponding filtrations coincide.
\end{Remark}

\medskip

Again, using spreading out, Theorem \ref{log-freeness} implies the following
theorem.  However, we also give a different proof that deduces it from
Theorem \ref{free-log-Simpson} that was already known in the non-logarithmic case ($D=0$).
Note also that Corollary \ref{loc-free-theorem}  can be proven in a somewhat simpler way than Theorem
\ref{log-freeness}. The difference is that if we follow the proofs of implications $(1)\Rightarrow (2)$ and $(1)\Rightarrow (3)$ in Theorem \ref{log-freeness}  under assumptions of Corollary \ref{loc-free-theorem}  then we do not need to consider the family $\{ G_m\} _{m\ge 0}$ and we can work directly with the family $\{ E_m\} _{m\ge 0}$.  However, it should be stressed that similar arguments as below (showing that Theorem \ref{0-mu} follows from Theorem \ref{free-log-Simpson})
do not allow to deduce Theorem \ref{log-freeness} from Corollary \ref{loc-free-theorem} in positive characteristic. 
This is caused by the use of
coverings that usually do not preserve liftability to
$W_2(k)$. Another problem is that such covers are sometimes
necessarily inseparable, in which case the pullback does not preserve
semistability.

\begin{Theorem} \label{0-mu}
  Let $X$ be a smooth projective variety of dimension $n\ge 2$ defined
  over a field of characteristic zero and let $D$ be a normal crossing
  divisor on $X$.  Let $H$ be an ample divisor on $X$ and let $(E,
  \theta)$ be a slope $H$-semistable logarithmic Higgs sheaf with
  $\Delta (E)H^{n-2}=0$.  If $E$ is reflexive then it is locally free
  and
  $$c_m(E)=\frac{\binom{r}{m}}{r^m} c_1(E)^m$$ 
  in $H^{2m}(X, \QQ )$ for all $m\ge 1$ and any $l\ne
  p$. Moreover, each rank $s$ factor $(G, \theta _G)$ of a slope
  $H$-Jordan-H\"older filtration of $(E, \theta)$ is locally free with
  $$c_m(G)=\frac{\binom{s}{m}}{r^m} c_1(E)^m$$ 
  in $H^{2m}(X, \QQ )$ for all $m\ge 1$.
\end{Theorem}

\begin{proof}
  By a variant of the Bloch--Gieseker covering trick (see
  \cite[Proposition 2.67]{KM}) there exists a smooth projective
  variety $\tilde X$ and a finite flat surjective covering $f: \tilde
  X\to X$ together with a line bundle $L$ such that $f^*(\det
  E)^{-1}=L^{\otimes r}$ and the pullback $\tilde D=(f^*D)_{\reduced}$
  is a simple normal crossing divisor. Let us define a logarithmic
  Higgs sheaf $(\tilde E, \tilde \theta: \tilde E\to \tilde E\otimes
  \Omega_{\tilde X}(\log \, \tilde D))$ by $(\tilde E, \tilde
  \theta):= f^*(E, \theta)\otimes L.$ Note that $\tilde E$ is
  reflexive, $c_1(\tilde E)=0$ and
$$\Delta (\tilde E)(f^*H)^{n-2}=\Delta (f^* E)(f^*H)^{n-2}=\deg f\cdot \Delta (E)H^{n-2}=0.$$ 
Hence $\tilde E$ is a slope $f^*H$-semistable logarithmic Higgs sheaf
with $\ch_1(\tilde E)(f^*H)^{n-1}=0$ and $\ch _2(\tilde
E)(f^*H)^{n-2}=0$. By Theorem \ref{free-log-Simpson} $\tilde E$ is
locally free and it has vanishing Chern classes.  Therefore by the
flat descent $E$ is also locally free and  we have
$$0=\Delta_m (\tilde E)=f^* (\Delta_m(E))$$
for all $m\ge 2$. 
Using the fact that $f$ induces an injection $H^{2m}(X, \QQ)\to H^{2m}(\tilde X,
\QQ)$, we get vanishing of $\Delta_m(E)$ for all $m\ge 2$.
Hence by Lemma \ref{equivalence-Delta} we get equalities
$r^mc_m(E)=\binom{r}{m}c_1(E)^m.$

Now let $(G, \theta _G)$ be a rank $s$ factor of a slope
$H$-Jordan-H\"older filtration of $(E, \theta)$. Then $f^*(G, \theta
_G)\otimes L$ is an extension of some factors of a slope
$f^*H$-Jordan-H\"older filtration of $(\tilde E, \tilde \theta)$. In
particular, it has a filtration whose quotients are locally free slope
$f^*H$-stable logarithmic Higgs sheaves with vanishing Chern
classes. It follows that $G$ is locally free, $c_1(f^*G)=-s
c_1(L)=\frac{s}{r}c_1(f^*E)$ and
$$c_m(G)=\frac{\binom{s}{m}}{s^m}c_1(G)^m=\frac{\binom{s}{m}}{r^m}c_1(E)^m$$
for all $m\ge 1$.
\end{proof}

\subsection{Restriction theorem}

The following theorem generalizes \cite[Theorem 12]{La2} to the logarithmic case
and to arbitrary $(Y,B)$.

\begin{Theorem} \label{curve-restriction} Let $X$ be a smooth
  projective variety of dimension $n$ defined over an algebraically
  closed field $k$ of characteristic $p$ and let $D$ be a normal
  crossing divisor on $X$.  Let $H$ be an ample divisor on $X$ and let
  $E$ be a locally free $\cO_X$-module of rank $r\le p$ with $\Delta
  (E) H^{n-2}=0$.  Assume that a logarithmic Higgs sheaf $(E, \theta)$
  is slope $H$-semistable.  Let $f: (Y, B)\to (X,D)$ be a proper
  morphism of smooth log pairs that has a good lifting to $W_2(k)$
  (see Definition \ref{good-lifting}).  Then the induced logarithmic
  Higgs sheaf $$f ^*(E, \theta)=(f^*E, f^*E\stackrel{f^*\theta}{\to}
  f^* E\otimes f^*\Omega_X(\log \, D) \stackrel{\id_{f^*E}\otimes df }
  {\longrightarrow} f^* E\otimes \Omega_Y(\log \, B ))$$ is slope
  $A$-semistable for any ample divisor $A$ on $Y$.
\end{Theorem}

\begin{proof} By Theorem \ref{deformation-to-system} we can deform
  $(E,\theta)$ to a slope $H$-semistable system of Hodge sheaves
  $(E_0,\theta_0)$. Moreover, by Theorem \ref{log-freeness} $E_0$ is
  locally free. If $f ^*(E_0, \theta_0)$ is semistable then by
  openness of semistability $f ^*(E, \theta )$ is also semistable.
  So without loss of generality one can assume that $(E,\theta)$ is a
  system of Hodge sheaves. The rest of the proof is the same as that
  of \cite[Theorem 12]{La2} using Theorem \ref{Higgs-de-Rham} instead
  of \cite[Theorem 5]{La2}. Here we also need to apply functoriality
  of the inverse Cartier transform in the logarithmic case (see
  Theorem \ref{functoriality-Cartier}).
\end{proof}

\medskip 

Applying the above theorem to iterates of the Frobenius morphism we get 
the following corollary:

\begin{Corollary}
  In the notation of the above theorem assume that $(Y,B)=(X,D)$ and
  $f$ is the Frobenius morphism.  Then $E$ is strongly $A$-semistable
  for any ample divisor $A$ on $X$.
\end{Corollary}

\medskip

In formulation of \cite[Theorem 12]{La2} the
author forgot to explicitly state the assumption on existence of
a compatible lifting of $C$ and $X$ (even though it was used in the
proof). The next example shows that this assumption is really
necessary.

\begin{Example}
  Here we show an example of a smooth projective surface that is
  liftable to the Witt ring $W(k)$ and a slope semistable Higgs
  sheaf, which is not semistable after restricting to
  the normalization of some projective curve on this surface.

  Let us consider a smooth complex projective surface $X$ which is a
  quotient of the product of upper half planes $\HH\times \HH$ by an
  irreducible, torsion free, cocompact lattice $G$ in $\PGL (2,
  \RR)\times \PGL(2,\RR)$.  Then $\Omega_X=L\oplus M$, where
  $L^2=M^2=0$, $L$ and $M$ are strictly nef (see \cite[Lemma 4.5]{La5}).
 
  Let us consider a Higgs bundle $(E, \theta)$, where $E=L\oplus
  \cO_X$ and $\theta$ is given by the canonical inclusion $L\to
  \Omega_X$. This Higgs bundle corresponds to the representation
  $\rho: \pi_1(X)\to \PGL (2, \CC)$ obtained by projecting the
  inclusion $G\subset \PGL(2, \RR)\times \PGL (2,\RR)$ onto the first
  factor and embedding into the complexification. Then the Higgs
  bundle $(E',\theta'):= \Sym ^2(E, \theta)\otimes (\det E, \det
  \theta)^{-1}$ corresponds to the composition of $\rho$ with the
  adjoint representation $ \PGL (2, \CC)\to \SL (3,\CC)$.  In
  particular, since this representation is irreducible, the Higgs
  bundle $(E', \theta')$ is slope stable (with respect to any
  polarization) and it has vanishing rational Chern classes.  This can
  be also checked directly from the definition of stability.
 More precisely,  $(E',\theta')$ is a system of Hodge sheaves $E^{2,0}\oplus E^{1,1}\oplus E^{0,2}=L\oplus \cO_X\oplus L^{-1}$
 with $\theta$ given by the canonical inclusions  $E^{2,0}=L\to E^{1,1}\otimes \Omega _X=L\oplus M $
and   $E^{1,1}=\cO_X\to E^{0,2}\otimes \Omega _X=\cO_X \oplus (M\otimes L^{-1})$ onto the first factor.
 This system has only two non-trivial saturated subsystems of Hodge sheaves given by 
 $E^{0,2}=L^{-1}$ and $ E^{1,1}\oplus E^{0,2}= \cO_X\oplus L^{-1}$. In particular,  $(E',\theta')$
 is slope $H$-stable if and only if $LH>0$.

  By openness of stability, the reduction of $(E', \theta')$ modulo
  almost all primes is stable.  Again this can be easily seen
  directly, because ampleness is an open condition and $LH>0$ implies an analogous inequality for the reductions. 
  Note also that for almost all reductions, $X_s$ lifts to $W(k(s))$.
 
  Now for a large number of primes (of positive density) the reduction
  of $L$ is not nef (see \cite[Example 5.6]{La5}). For such $s$ there exists an irreducible curve
  $C_s$ such that $L_s.C_s<0$. Let $\nu _s: \bar C_s\to C_s$ be the
  normalization. Then $\nu _s^*(E'_s,\theta'_s)$ is not semistable
  because it has degree zero and it contains a Higgs subbundle $(\nu
  _s^*L^{-1}_s,0)$ of positive degree. This shows that $\nu_s: \bar
  C_s\to X_s$ cannot be compatibly lifted to $W_2(k(s))$, even though
  both $\bar C_s$ and $X_s$ can be lifted to $W(k(s))$.
\end{Example}

By the usual spreading-out technique, Theorem \ref{curve-restriction}
implies the following corollary. 

\begin{Corollary}\label{curve-restriction-0}
  Let $X$ be a smooth projective variety of dimension $n$ defined over
  an algebraically closed field $k$ of characteristic $0$. Let $D$ be a normal
  crossing divisor on $X$ and let  $H$ be
  an ample divisor on $X$. Let $E$ be a locally free $\cO_X$-module
  with $\Delta (E) H^{n-2}=0$.  If a logarithmic Higgs sheaf $(E, \theta)$ is slope
  $H$-semistable then for every smooth projective curve $C$ not contained in $D$ 
  and a morphism $f: C\to X$ the Higgs bundle $f^*(E, \theta)$ is
  semistable.
\end{Corollary}

\begin{Remark}
  In the case of complex projective manifolds and $D=0$  the above corollary follows  
  from Simpson's correspondence. A rough sketch of proof is as follows.
  A slope semistable Higgs bundle with 
  vanishing rational  Chern classes corresponds to a local system on $X$. So for any morphism $f : C\to
  X$ we get an induced local system on $C$. This again corresponds to
  a slope semistable Higgs bundle on $C$. By functoriality of
  Simpson's correspondence this is the pullback of the original Higgs
  bundle. The general case with $\Delta (E)H^{n-2}=0$ can be reduced
  to the above one by taking $\mathop{\rm End\,} E$ and using \cite[Theorem 2]{Si}
  (or Theorem \ref{free-log-Simpson}). More precisely, if  $\Delta (E)H^{n-2}=0$ then $c_1  (\mathop{\rm End\,} E)=0$ and
  $\Delta (\mathop{\rm End\,} E) H^{n-2}=0$, so $ (\mathop{\rm End\,} E , \theta _{\mathop{\rm End\,} E})$  is semistable with vanishing rational Chern classes.
  Then  $f^*(\mathop{\rm End\,} E, \theta _{\mathop{\rm End\,} E})$ is semistable, which implies
  semistability of $f^*(E, \theta)$.
\end{Remark}

\section{Nearby-cycles}

The main aim of this section is to understand the restriction of an integrable
logarithmic connection (or a logarithmic Higgs sheaf) to the boundary divisor. 
In case of Hodge structures on complex varieties the
analogous problem is realised by the construction of a nearby-cycle
functor for the category of real graded-polarized families of mixed
Hodge structures (see \cite[Section 4]{Br1}). Here we use a different
approach that allows us to keep more information about the
restrictions.  As in \cite{Br1} this construction is related to the
standard constructions of a nearby-cycles functor coming back to
Grothendieck, Deligne and Saito.

In this section we will use some basic facts and definitions related to Lie algebroids
for which we refer to \cite{La4}.

\subsection{Nearby-cycles functor} \label{def-Lie}

Let $X$ be a smooth projective variety of dimension $n$ defined over
an algebraically closed field $k$. Let $D$ be a simple normal
crossing divisor on $X$ and let $Y$ be an irreducible component of $D$.

\medskip

Let $\imath: Y\hookrightarrow X$ be the canonical embedding.  We
define a Lie algebroid $\LL _Y$ on $Y$ as the triple $(L, [\cdot,\cdot
], \alpha)$, where $L=\imath^*T_X(\log D)$ is a locally free
$\cO_Y$-module with the Lie algebra structure induced from the
standard Lie algebra structure on $T_X$ and the anchor map $\alpha:
L\to \mathop{\rm Der}_k(\cO_Y)=T_Y$ is the canonical map induced by
$\imath$. The anchor map induces a $k$-derivation
$d_{\Omega_{\LL_Y}}:\cO_Y\to \Omega_{\LL_Y}=L^*$.

Giving an $\LL_Y$-module structure ${\LL_Y}\to \End _kE$ on a coherent
$\cO_Y$-module $E$ is equivalent to giving an integrable
$d_{\Omega_{\LL_Y}}$-connection $\nabla_{\LL_Y}: E\to E\otimes
_{\cO_Y}\Omega_{\LL_Y}$ (see \cite[Lemma 3.2]{La4}).

The usefulness of the above construction comes from the fact that the
restriction to $Y$ defines an obvious functor 
$$\Psi _Y: \Mic (X, D)\to  \Mod{\LL _Y}$$ 
from the category $\Mic (X, D)$ of coherent $\cO_X$-modules with an
integrable logarithmic connection on $(X, D)$ to the category
$\Mod{\LL _Y}$ of coherent $\cO_Y$-modules with an $\LL _Y$-module
structure. If $(V, \nabla)$ is a coherent $\cO_X$-module with an integrable
connection on $(X, D)$ then $\Psi_Y(V, \nabla)$ is defined as the
restriction $\imath^* V$ of $V$ to $Y$ and the $\LL_Y$-module
structure is given by the integrable $d_{\Omega_{\LL_Y}}$-connection
$\imath^* \nabla : \imath^* V \to \imath^* V\otimes \imath^* \Omega _X
(\log \, D)$. By an abuse of notation we will often write $(V_Y,\nabla|_Y)$
to denote $\Psi_Y(V, \nabla)$.

\medskip

Let $\LL_Y^0$ be the trivial Lie algebroid underlying $\LL _Y$ (i.e.,
we consider the same $L=\imath^*T_X(\log D)$ but with zero Lie bracket
and zero anchor map). Similarly as above, we get the functor
$$\Phi _Y: \Hig (X, D)\to  \Mod{\LL _Y^0}$$ 
from the category $\Hig (X, D)$ of coherent $\cO_X$-modules with a
logarithmic Higgs field on $(X, D)$ to the category $\Mod{\LL _Y^0}$
of coherent $\cO_Y$-modules with an $\LL _Y^0$-module structure. Note
that $\Mod{\LL _Y^0}$ is the same as the category of coherent
$\cO_Y$-modules with a $\Sym ^{\bullet} (\imath^* T_X (\log \,
D))$-module structure. Similarly, as above we will often write
$(E_Y,\theta|_Y)$ to denote $\Phi(E,\theta)$.

\subsection{General monodromy filtrations}

Let $Y$ be a smooth projective variety of dimension $n$ defined over
an algebraically closed field $k$.  
Let $\LL$ be a smooth Lie algebroid on $Y/k$ and let $E$ be an
$\LL$-module.

Let $N: E\to E$ be a nilpotent endomorphism of $\LL$-modules.  By
\cite[Proposition 1.6.1]{De} $N$ induces on $E$ a unique finite
increasing filtration $M_{\bullet}$ by ${\LL}$-submodules such that:
\begin{enumerate}
\item $N (M_i)\subset M_{i-2}$ for all $i$,
\item $N^i$ induces an isomorphism $\Gr ^M_iE\stackrel{\sim}{\to} \Gr ^M_{-i}E$ for all $i\ge 0$.
\end{enumerate}
We call $M_{\bullet}$ the \emph{monodromy filtration} for the ${\LL}$-module $E$.

\medskip

Let us define the $j$-th primitive part $P_j(E)$ of $E$ as the kernel
of $N^{j+1}: \Gr ^{M}_jE\to \Gr ^{M}_{-j-2}E$ for $j\ge 0$ and
$P_j(E)=0$ for $j<0$. Then by \cite[(1.6.4)]{De} we have the
decomposition into primitive parts
\begin{equation}\label{prim-deco}
 \Gr ^{M}_jE=\bigoplus _{i\ge \max (0, -j)}N^iP_{j+2i}(E)\simeq \bigoplus _{i\ge \max (0, -j)}P_{j+2i}(E).
\end{equation}

\begin{Lemma}\label{torsion-free-quotients}
If $E$ is torsion free (as an $\cO_Y$-module) then all quotients $\Gr ^M_jE$
are also torsion free.
\end{Lemma}

\begin{proof}
  The proof is by induction on the rank of $E$. If $E$ has rank $1$
  then $N$ is nilpotent if and only if $N=0$, so the filtration is trivial.

  Now let us assume that the assertion holds for all sheaves of rank
  less than the rank of $E$. If $N=0$ the assertion is trival, so we can assume that $N\ne 0$.
  Since for $j\ge 0$ the map $N^j$ induces
  an isomorphism $\Gr ^M_jE\stackrel{\sim}{\to} \Gr ^M_{-j}E$, the
  image $N^jP_j(E)$ is the kernel of $N: \Gr ^{M}_{-j}E\to \Gr
  ^{M}_{-j-2}E$.  By \cite[Corollaire (1.6.6)]{De} the associated
  graded of the filtration induced by $M_{\bullet}$ on $\ker N$
  satisfies
$$\Gr _{-j}^M(\ker N)\stackrel{\sim}{\to}N^jP_j(E)\simeq P_j(E).$$
But $\ker N\subset E$ is torsion free and since $N$ is nilpotent, the rank of 
$\ker N$ is less than the rank of $E$.  So by the induction assumption all
quotients $\Gr _{-j}^M(\ker N)$ are torsion free. Hence all $P_j(E)$
are torsion free and by the decomposition (\ref{prim-deco}) all $\Gr
^M_jE$ are also torsion free.
\end{proof}

Now let us fix an ample divisor $H$ on $Y$. If an $\LL$-module $E$ is
slope $H$-semistable then we always assume that it is torsion free as
an $\cO_Y$-module. The following lemma proves that the monodromy
filtration (or the filtration by primitive cohomology) of a slope
$H$-semistable $\LL$-module can be always refined to a
Jordan--H\"older filtration.

\begin{Lemma} \label{semistable-quotients}
  Let $E$ be a slope $H$-semistable $\LL$-module. Then every quotient
  $\Gr _j^ME$ of the monodromy filtration $M_{\bullet}$ of $E$ is
  slope $H$-semistable with $\mu _{H}(\Gr _j^ME)=\mu _{H}(E)$.
Moreover, all  $P_j(E)$ are  slope $H$-semistable with $\mu _{H}(P_j(E))=\mu _{H}(E)$.
\end{Lemma}

\begin{proof} 
  The proof is by induction on the rank of $E$. For rank $1$ the
  assertion is clear so assume that it holds for all sheaves of rank
  less than the rank of $E$.

  Let $d$ be the largest integer such that $M_{-d}\ne 0$.  Since $E$
  is slope $H$-semistable we have $\mu _{H}(M_{-d})\le \mu _{H}(E)$.
  But $N^d$ induces an isomorphism $\Gr ^M_{d}E=E/M_{d-1}\stackrel{\sim}{\to}\Gr
  _{-d}^ME= M_{-d}$ and by slope
  $H$-semistability of $E$ we get $\mu _{H}(M_{-d})=\mu
  _{H}(E/M_{d-1})\ge \mu _{H}(E)$.  Hence $\mu _{H}(M_{-d})= \mu
  _{H}(E)$ and $M_{-d}$ is slope $H$-semistable. So 
  $E/M_{d-1}\simeq M_{-d}$ is also slope $H$-semistable with 
  $\mu _{H}(E/M_{d-1})=\mu _{H}(E)$.
  This shows that $M_{d-1}$ is  slope $H$-semistable with
  $\mu _{H}(M_{d-1})=\mu_{H}(E)$.   Note also that $M_{d-1}/M_{-d}$ is torsion free by Lemma
  \ref{torsion-free-quotients}.  Since  $\mu _{H}(M_{-d})=\mu _{H}(M_{d-1})=\mu _{H}(E)$, 
  this implies that $M_{d-1}/M_{-d}$ is slope $H$-semistable with $\mu _{H}(M_{d-1}/M_{-d})=\mu _{H}(E)$.
  But $N$ induces on $M_{d-1}/M_{-d}$ a nilpotent endomorphism whose
  quotients coincide with the remaining quotients of the monodromy
  filtration $M_{\bullet}$ of $E$. Hence by the induction assumption
  all $\Gr _j^ME$ are slope $H$-semistable with $\mu _{H}(\Gr
  _j^ME)=\mu _{H}(E)$.

  The second assertion follows immediately from the first one and the
  decomposition (\ref{prim-deco}) of $\Gr _j^ME$ into primitive parts.
\end{proof}

\subsection{Residue maps} 

Let $X$ be a smooth projective variety of dimension $n$ defined over
an algebraically closed field $k$. Let $D$ be a simple normal
crossing divisor on $X$ and let $Y$ be an irreducible component of $D$.
We can write $D=D'+Y$ for some divisor $D'$ which does not contain $Y$.
In the following we denote $D'$ by $D-Y$ and set $D^Y=(D-Y)|_Y$. 

\medskip

Note  that ${\LL_Y}$ (see Subsection \ref{def-Lie})  is equipped with the canonical map
$\Res : \Omega_{\LL_Y}=\imath^*\Omega_X(\log \, D)\to \cO_Y$ given by the Poincar\'e residue.
Using it for any ${\LL_Y}$-module $E$ we can define the \emph{residue endomorphism}  $\Res_E$
as a composition
$$E\stackrel{\nabla_{\LL_Y}}{\longrightarrow} E\otimes _{\cO_Y} \Omega_{\LL_Y} \stackrel{\id_E\otimes \Res}{\longrightarrow} E\otimes_{\cO_Y} \cO_Y=E.$$ 
Since $\Res\circ d_{\Omega_{\LL_Y}}=0$, this endomorphism is
$\cO_Y$-linear.  It is easy to check that $\Res_E$ is an endomorphism
of ${\LL_Y}$-modules.  In the same way we can define the residue
endomorphism of an ${\LL_Y^0}$-module.

\medskip

Let $\Mod{\LL_Y}_0$ ($\Mod{\LL_Y^0}_0$) be the full subcategory of $\Mod{\LL_Y}$
  ($\Mod{\LL_Y^0}$) containing as objects all $\LL_Y$-modules $E$
  ($\LL_Y^0$-modules, respectively) with $\Res
  _E=0$. 
  
 Similarly, let $\nilMod{\LL_Y}$ ($\nilMod{\LL_Y^0}$) be the full subcategory of $\Mod{\LL_Y}$
  ($\Mod{\LL_Y^0}$) containing as objects all $\LL_Y$-modules $E$
  ($\LL_Y^0$-modules, respectively) that have nilpotent residue $\Res
  _E$. 

\begin{Lemma} \label{functor-Upsilon}
The category $\Mod{\LL_Y}_0$ is equivalent to the category $ \Mic (Y,D^Y)$. Similarly, 
the category $\Mod{\LL^0_Y}_0$ is equivalent to the category $ \Hig (Y,D^Y)$.
Moreover, we have natural functors
$$\Upsilon: \nilMod{\LL_Y}\to \Mic (Y,D^Y)$$
given by sending $E$ to $\Gr ^WE$, where
$W_{\bullet}$ is the monodromy
filtration of $\Res _E$
and 
$$\Upsilon ^0: \nilMod{\LL_Y^0}\to \Hig (Y,D^Y)$$
given by sending $E$ to $\Gr ^ME$, where
$M_{\bullet}$ is the monodromy
filtration of $\Res _E$.
\end{Lemma}

\begin{proof}
The short exact sequence
$$0\to \Omega_Y (\log \, D^Y)\to \imath^*\Omega_X(\log \, D)\mathop{\to}^{\Res } \cO_Y \to 0.$$
shows that an ${\LL_Y}$-module $E$ with $\Res_E=0$ gives rise to a canonically
defined integrable logarithmic connection  $E\to
E\otimes_{\cO_Y} \Omega_Y (\log \, D^Y)$. Conversely, if $(V, \nabla)$ is an element of $\Mic (Y, D^Y)$
then $\nabla$ defines an integrable $d_{\Omega_{\LL _Y}}$-connection, so we get an ${\LL_Y}$-module $V$ with $\Res_V=0$. If $E$ is an ${\LL_Y^0}$-module  with $\Res_E=0$ then the same argument shows that
$E$ is a logarithmic Higgs sheaf on $(Y, D^Y)$. This shows the first part of the lemma.

Now let us assume that $E$ is an ${\LL_Y}$-module with nilpotent
$N=\Res _E$. Let $W_{\bullet}$ be the corresponding monodromy
filtration (in the category of ${\LL_Y}$-modules).  Note that the
composition $W_i\stackrel{\Res _{W_i}}{\longrightarrow} W_i\to
W_i/W_{i-1}$ is zero as $N (W_i)\subset W_{i-1}$.  Hence $\Res _{\Gr
  ^W_iE}=0$ and each quotient $\Gr ^W_iE$ is endowed with an integrable
logarithmic connection $\nabla_i^W$ on $(Y, D^Y)$. Similarly, for an
${\LL_Y^0}$-module $E$ with nilpotent $N=\Res_E$ all quotients $\Gr
_i^ME$ of the monodromy filtration $M_{\bullet}$ have canonically
defined structure of a logarithmic Higgs sheaf $(\Gr _i^ME,
\theta_i^M)$ on $(Y, D^Y)$.
\end{proof}

\medskip

Let $D_i$ be an irreducible component of $D$ different from $Y$.  Let
$\Res _{D_i}: \Omega_X(\log \, D)\to \cO_{D_i}$ be the Poincar\'e
residue along $D_i$. Pulling it back to $Y$ we get an $\cO_Y$-linear
map $\Res _{Y}^{D_i}: \Omega_{\LL_Y}=\imath^*\Omega_X(\log \, D)\to
\cO_{D_i^Y}$, where $D_i^Y={D_i\cap Y}$. Now for any $\LL _Y$-module $E$ we consider
the composition map
$$E\stackrel{\nabla_{\LL_Y}}{\longrightarrow} E\otimes _{\cO_Y} \Omega_{\LL_Y} \stackrel{\id_E\otimes 
  \Res _{Y}^{D_i}}{\longrightarrow} E\otimes_{\cO_Y}
\cO_{D_i^Y}=E_{D_i^Y}.$$ One can easily check that this map is
$\cO_Y$-linear and it factors through the restriction map $E\to
E_{D_i^Y}$. Therefore it defines the map $\Res_{E}^{D_i}: E_{D_i^Y}\to
E_{D_i^Y}$ that we call \emph{the residue map of $E$ along $D_i$}.  In
the same way we can define the residue maps along $D_i$ for any $\LL
_Y^0$-module $E$.

\medskip
\begin{Remark}
Let $(V, \nabla)$ be an object of $MIC(X, D)$ and let $E=\Psi _Y (V, \nabla)$ be the
corresponding $\LL_Y$-module. Then the residue map
$\Res_ E: E\to E$ coincides with the residue map $\Res _{Y}\nabla:
V_{Y}\to V_{Y}$.  Similarly, for any irreducible component $D_i$
of $D-Y$ the residue map $\Res_{E}^{D_i}: E_{D_i^Y}\to E_{D_i^Y}$
coincides with the restriction of the residue map $\Res _{D_i}\nabla:
V_{D_i}\to V_{D_i}$ to $D_i^Y$.
\end{Remark}

\subsection{Compatibility of the Cartier transform with monodromy filtrations}

Let $X$ be a smooth projective variety of dimension $n$ defined over
an algebraically closed field $k$ and let $D$ be a simple normal
crossing divisor on $X$. Let $Y$ be an irreducible component of $D$.

Let $Z=\VV (\cO_Y(-Y))$ be the total space of the normal bundle of
$\imath: Y\hookrightarrow X$ and let $\pi : Z\to Y$ be the canonical
projection. Let $s:Y\to Z$ be the zero section and let $Y_0$ be its image.

\begin{Lemma}\label{Wahl}
Let us set $D^Z= Y_0+\pi^{-1}(D^Y)$.
The short exact sequence
$$0\to \pi ^*\Omega_Y (\log \, D^Y)\to \Omega _{Z} (\log \, D^Z)
{\to} \Omega _{Z/Y} (\log \, Y_0) =\cO_Z\to 0.$$
is the pull back of 
$$0\to \Omega_Y (\log \, D^Y)\to \imath ^*\Omega_X(\log \, D)\mathop{\to}^{\Res _Y} \cO_Y \to 0.$$
\end{Lemma}

\begin{proof}
Let us recall that the extension class of
$$0\to \Omega_Y \to \imath ^*\Omega_X(\log \, Y)\mathop{\to}^{\Res _Y} \cO_Y \to 0$$
in $\Ext^1(\cO_Y, \Omega_Y)=H^1(\Omega_Y)$ is equal to the Atiyah
class of $\cO_Y(-Y)$, which is also the image of the class of $\cO_Y
(-Y)$ in $H^1(\cO_Y^*)$ under the map $H^1(\cO_Y^*)\to H^1(\Omega_Y)$.
Hence by \cite[Proposition 3.3]{Wa} the pull back of the above sequence to $Z$ induces
$$0\to \pi ^*\Omega_Y \to \Omega _{Z} (\log \, Y_0) {\to} \cO_Z\to 0.$$
Let $\{D_i\}$ be the divisors corresponding to the irreducible components of $D-Y$.
Now the required assertion follows from the following standard exact sequences:
$$0\to \Omega_Y \to \Omega_Y(\log \, D^Y) \to \bigoplus \cO_{D_i\cap Y} \to 0,$$
$$0\to \imath ^*\Omega_X(\log \, Y) \to \imath ^*\Omega_X(\log \, D) \to \bigoplus \cO_{D_i\cap Y} \to 0,$$
and
$$0\to  \Omega _{Z} (\log \, Y_0)\to  \Omega _{Z} (\log \, D^Z) \to \bigoplus \cO_{\pi^{-1}(D_i\cap Y)} \to 0.$$
\end{proof}

An alternative proof of the lemma can be obtained, e.g., by directly making local
calculation and checking equality of the corresponding gluing
conditions (cf. proof of \cite[Proposition 3.3]{Wa}).

\medskip

Let $(V, \nabla)$ be a coherent $\cO_X$-module with an integrable logarithmic
connection $\nabla: V\to V\otimes \Omega_X(\log \, D)$.
After restricting to $Y$ we see that $V_Y$ acquires an integrable $\Omega_{\LL _Y}$-connection. 
After further pull back to $Z$ we get an induced integrable logarithmic connection 
$$\nabla': \pi ^* V_Y\to  \pi^* V_Y  \otimes \Omega _{Z} (\log \, D^Z).$$
The same construction allow us to associate to a logarithmic Higgs
sheaf $(E, \theta)$ on $(X, D)$, a logarithmic Higgs sheaf $(\pi^*
E_Y, \theta')$ on $(Z, D^Z)$. Note that if $\theta$ is
nilpotent then  $\theta'$ is also nilpotent.

\medskip

\begin{Remark}
  One could naively hope that one can work with logarithmic
  connections on projective varieties by pulling back the $\LL
  _Y$-module $(V_{Y}, \nabla |_Y)$ via $\varphi : T=\PP (\cO_Y(-Y)
  \oplus \cO_Y)\to Y$.  Indeed, one has a short exact sequence
$$0\to \varphi ^*\Omega_Y (\log \, D^Y)\to \Omega _{T} (\log \, Y_0+Y_{\infty}+\pi^{-1}(D^Y))
{\to} \Omega _{T/Y} (\log \, Y_0+Y_{\infty}) =\cO_T\to 0,$$
where $Y_{\infty}=T-Z$ is image of the  infinity section. But 
if $p\ne 2$ then
$$0\to \varphi ^*\Omega_Y (\log \, D^Y)\to \varphi^* (\imath ^*\Omega_X(\log \, D)) \to \cO_T\to 0$$
defines a different extension class. This can be seen by computing the
extension class of both sequences after restricting to
$Y_{\infty}$. This forces us to deal with non-projective varieties,
where the difficulty is that one cannot directly apply Theorem
\ref{log-freeness}.
\end{Remark}

\medskip

Let us assume that the base field $k$ has characteristic $p$ and
$(X,D)$ is liftable to $W_2(k)$. Let us fix a lifting $(\tilde X,
\tilde D)$.  This lifting induces a lifting $(\tilde Y, \tilde D^Y)$
of $(Y, D^Y)$ to $W_2(k)$ and also a compatible lifting $(\tilde Z,
\tilde D^Z)$ of $(Z, D^Z)$ to $W_2(k)$.

The following lemma is functoriality of Cartier transforms in a
situation that is not covered by Theorem \ref{functoriality-Cartier}
(we do not even have a map $(Z, D^Z)\to (X,D)$).

\begin{Lemma} \label{compatibility}
  Let $(E, \theta)$ be a reflexive logarithmic Higgs sheaf on
  $(X,D)$ with a nilpotent Higgs field  of level less or equal to $p-1$. If $(V, \nabla)= C_{(\tilde
    X,\tilde D)}^{-1} (E, \theta)$ then we have a canonical isomorphism $(\pi ^* V_Y,  \nabla ')\simeq C_{(\tilde Z,  \tilde D^Z)}^{-1} (\pi^*E _Y, \theta ')$
and the diagram
$$\xymatrix{(\pi ^* V_Y,  \nabla ')\ar[d]^{\pi^*\Res_Y\nabla}
\ar[r]^-{\simeq}& C_{(\tilde Z,  \tilde D^Z)}^{-1} (\pi^*E _Y, \theta ')\ar[d]^{C_{(\tilde Z, \tilde D^Z)}^{-1}(\pi^*\Res _Y\theta)} \\
(\pi ^* V_Y,  \nabla ')\ar[r]^-{\simeq}& C_{(\tilde Z,  \tilde D^Z)}^{-1} (\pi^*E _Y, \theta ')\\
}$$
is commutative.
\end{Lemma}

\begin{proof}
  Since $E$ is reflexive and $Y$ is smooth, $E_Y$ is torsion
  free. Since $\pi$ is flat, $\pi^*E_Y$ is also torsion free, so we
  can apply $C_{(\tilde Z, \tilde D^Z)}^{-1}$ to $(\pi^*E _Y, \theta')$.
We will use the notation introduced in proof of Theorem \ref{functoriality-Cartier}.

There exist an affine covering $\{\tilde U_{\alpha}\} _{\alpha\in I}$
of $\tilde X$ such that for each $\tilde U_{\alpha}$ we have a system
of logarithmic coordinates, i.e., $x_1,...,x_n$ such that $\tilde
D\cap \tilde U_{\alpha}$ is given by $\prod _{i=1}^{n_0}x_i=0$, with
$x_1=0$ giving $\tilde Y\cap \tilde U_{\alpha}$.  We can assume that
$\cO_{\tilde U_{\alpha}}(-\tilde Y)$ is trivial and choose for each
$\alpha$ its generator $t$.  Let us also choose standard log Frobenius
liftings $\tilde F_{U_{\alpha}}: \tilde U_{\alpha} \to \tilde
U_{\alpha}$ so that $\tilde F_{U_{\alpha}}^*(x_i)=x_i^p$.  
Then the projection $\tilde V_{\alpha}:= \pi^{-1}(\tilde U_{\alpha}
\cap \tilde Y)\to \tilde U_{\alpha} \cap \tilde Y$ corresponds to the
projection $(\tilde U_{\alpha} \cap \tilde Y)\times _{W_2(k)}\Spec
W_2(k)[t]\to \tilde U_{\alpha} \cap \tilde Y$.  We choose a
logarithmic Frobenius lifting of $(\tilde V _{\alpha}, \tilde D^Z\cap
\tilde V_{\alpha})$ to be $\tilde F_{V_{\alpha}} =\tilde F_{\tilde U_{\alpha} \cap \tilde
  Y}\times \tilde F_{\Spec k[t]}$, where $ \tilde F_{\Spec k[t]}$ is
given by $t\to t^p$. Note that $ \tilde D^Z\cap \tilde V_{\alpha}$ is
given by $\prod _{i=1}^{n_0}x_i t=0$.  We can locally write
$$\theta|_{U_{\alpha}}= \sum _{i=1}^{n_0} \theta_i\otimes \frac{dx_i}{x_i}+\sum _{i=n_0+1}^{n} \theta_i\otimes {dx_i}, $$
where $\theta_i: E_{U_{\alpha}}\to E_{U_{\alpha}}$ are some commuting endomorphisms.
This allows us to identify $C_{(\tilde Z,  \tilde D^Z)}^{-1} (\pi^*E _Y, \theta ')$.
Over each $\tilde V_{\alpha}$ we have $F^*(\pi^*(E_{U_\alpha\cap Y}))$ with the connection given by 
$$\nabla_{\alpha} := \nabla_{can}+({\id} \otimes \zeta_{\alpha})\circ (F^*\pi^*(\theta|_Y)),$$
where $\zeta_{\alpha}=\frac{d\tilde F_{V_{\alpha}}}{p}$. The isomorphism
from Lemma \ref{Wahl} is locally given by $\pi^*(\frac{dx_1}{x_1}|_Y)= \frac{dt}{t}$,
$\pi^*(\frac{dx_i}{x_i}|_Y)= \frac{dx_i}{x_i}$ for $2\le i\le n_0$ and $\pi^*({dx_i}|_Y)= dx_i$
for $n_0\le i\le n$. So we get 
$$\nabla_{\alpha} := \nabla_{can}+  F^*\pi^*(\theta_1|_Y)\otimes \frac{dt}{t}+
\sum _{i=2}^{n_0}F^* \pi^*(\theta_i|_Y)\otimes \frac{dx_i}{x_i}+\sum _{i=n_0+1}^{n} F^* \pi^*(\theta_i|_Y)\otimes 
x_i^{p-1}{dx_i}.$$
On the other hand, locally on $U_{\alpha}$, $(V, \nabla)$ can be identified with $F^*E_{U_\alpha}$
 with the connection given by 
$$\nabla|_{U_{\alpha}} := \nabla_{can}+({\id} \otimes \zeta_{\alpha}')\circ (F^*\theta),$$
where $\zeta_{\alpha}'=\frac{d\tilde F_{U_{\alpha}}}{p}$. Writing down this formula in local coordinates we get
$$\nabla|_{U_{\alpha}} = \nabla_{can}+\sum _{i=1}^{n_0}F^*\theta_i\otimes \frac{dx_i}{x_i}+
\sum _{i=n_0+1}^{n} F^* \theta_i\otimes x_i^{p-1}{dx_i}.$$ Using
equality $\nabla_{can}=\pi^*(\nabla_{can}|_{Y})$ and the above formulas
we get $\nabla_{\alpha} =\pi^*(\nabla|_{U_{\alpha}\cap Y})$. Checking
equality of gluing conditions is similar and left to the reader.

Finally, note that since the isomorphism $(\pi ^* V_Y,  \nabla ')\simeq 
C_{(\tilde Z,  \tilde D^Z)}^{-1} (\pi^*E _Y, \theta ')$
is functorial with respect to open embeddings $V_{\alpha}\subset Z$, it is sufficient to check the commutativity of 
the diagram only locally. In the local situation this follows easily from local equalities
$\Res_Y\nabla |_{U_{\alpha}}= F^*(\Res _Y\theta|_{U_{\alpha}})$.
\end{proof}

Let $(E, \theta)$ be a logarithmic Higgs sheaf on $(X,D)$ with a
nilpotent Higgs field.  Let $M_{\bullet}$ be the monodromy filtration
for $\Res _Y\theta$. Then each quotient $\Gr ^M_iE_Y$ is endowed with a
nilpotent logarithmic Higgs field $\theta_i^M$ on $(Y, D^Y)$.  

Let $(V, \nabla)$ be an object of $\Mic(X,D)$. Assume that the residue
$\Res _Y(\nabla)$ is nilpotent and let $W_{\bullet}$ be the monodromy
filtration for $\Res _Y\nabla$. Then each quotient $\Gr ^{W}_iV_Y$ is
endowed with a nilpotent integrable logarithmic connection $\nabla_i^W$ on $(Y,
D^Y)$.

\begin{Proposition}
  Let $(E, \theta)$ be a reflexive logarithmic Higgs sheaf on $(X,D)$
  with a nilpotent Higgs field of level less or equal to $p-1$ and let
  $(V, \nabla)= C_{(\tilde X,\tilde D)}^{-1} (E, \theta)$. Let
  $M_{\bullet}$ be the monodromy filtration for $\Res _Y\theta$ and
  let $W_{\bullet}$ be the monodromy filtration for $\Res _Y\nabla$.
  Then $(\Gr ^M_iE_Y, \theta_i^M)$ is a torsion free logarithmic Higgs
  sheaf on $(Y,D^Y)$ with a nilpotent Higgs field of level less or
  equal to $p-1$ and we have
$$(\Gr ^{W}_iV_Y, \nabla_i^W)=C_{(\tilde Y,\tilde D^Y)}^{-1}(\Gr ^M_iE_Y, \theta_i^M).$$ 
\end{Proposition}

\begin{proof}
  Note that $\pi^*M_{\bullet}$ is a filtration of $(\pi^*E _Y, \theta
  ')$ by logarithmic Higgs submodules on $(Z, D^Z)$. Moreover,
  quotients of this filtration are logarithmic Higgs modules on $(Z,
  \pi^{-1}(D^Y))$. Similarly, $\pi^*W_{\bullet}$ is a filtration of
  $(\pi^*V _Y, \nabla ')$ by integrable logarithmic connections on
  $(Z, D^Z)$ and the quotients are objects of $\Mic (Z,
  \pi^{-1}(D^Y))$.

Lemma \ref{compatibility} and uniqueness of the monodromy filtrations imply that
$$(\pi ^* W_i,  \nabla _i')= C_{(\tilde Z,  \tilde D^Z)}^{-1} (\pi^*M_i, \theta _i'),$$
where $\nabla'_i$ and $\theta'_i$ denote the restriction of $\nabla'$
and $\theta'$ to the corresponding subsheaves.
But this implies that
$$ \pi^* (\Gr ^{W}_iV_Y, \nabla_i^W) =  C_{(\tilde Z,  \tilde D^Z)}^{-1} \pi^* (\Gr ^M_iE_Y, \theta_i^M)=
C_{(\tilde Z, \widetilde {\pi^{-1}(D^Y)} }^{-1} \pi^* (\Gr ^M_iE_Y,
\theta_i^M).$$ Pulling back this equality by the zero section $s: (Y, D^Y)\to
(Z, \pi^{-1}(Y))$ and using functoriality of the Cartier transform, we get the
required assertion.
\end{proof}

\subsection{Nearby-cycles in positive characteristic}

Let $X$ be a smooth projective variety of dimension $n$ defined over
an algebraically closed field $k$ of characteristic $p$ and let $D$ be
a simple normal crossing divisor on $X$. In this subsection we 
assume also that $(X,D)$ is liftable to $W_2(k)$ and we fix a lifting $(\tilde
X, \tilde D)$.

\medskip

Let $H$ be an ample divisor on $X$ and let us fix a class $\mu \in
H^{2}_{\et} (X, \QQ_l)$ for some $l\ne p$.  We define the \emph{category $\HIG{\mu}
(X,D)$ of minimally semistable Higgs sheaves of slope $\mu$}
as the full subcategory of the category $\Hig (X, D)$ of logarithmic Higgs
sheaves on $(X,D)$, whose whose objects are pairs $(E,\theta)$, where
\begin{itemize}
\item $E$ is a locally free  $\cO_X$-module of rank $r\le p$, 
\item  $(E, \theta)$ is slope $H$-semistable,
\item  $c_1(E)=r\mu$ (i.e.,  the slope of $E$ is equal to $\mu$),
\item $\Delta (E) H^{n-2}=0$ (i.e., $E$ has a minimal possible discriminant).
\end{itemize}

By Theorem \ref{log-freeness} for any object $(E, \theta)$ of $\HIG{\mu}
(X,D)$ we have $c_m(E)=\binom{r}{m}\mu ^m$ for all $m\ge 1$.
Taking in Theorem \ref{curve-restriction} as $f$ identity, we see that
the above category does not depend on the choice of polarization $H$.

Unfortunately, $\HIG{\mu} (X,D)$ is not abelian as it does not contain
direct sums of objects.  However, by Theorem \ref{log-freeness} it
satisfies all other axioms of the abelian category.  In particular, it
contains kernels, images and cokernels (cf. \cite[Corollary 5]{La3})
and any morphism in this category admits a canonical decomposition.

\medskip

\medskip

Let $Y$ be an irreducible component $D$ and let us fix a class $\eta \in
H^{2}_{\et} (Y, \QQ_l)$ for some $l\ne p$. Let us define the category
$\ssMod{\LL_Y^0}{\eta}$ as the full subcategory of the category
$\Mod{\LL _Y^0}$ (defined in Subsection \ref{def-Lie}), whose objects $E$ satisfy the following conditions:
\begin{itemize}
\item as an $\cO_Y$-module $E$ is locally free of rank $r\le p$, 
\item  $E$ is slope $H_Y$-semistable (as an $\LL_Y^0$-module),
\item  $c_1(E)=r\eta$ and  $\Delta (E) H_Y^{n-3}=0$.
\end{itemize}
Replacing in the above definition $\LL_Y^0$  by $\LL_Y$
one can also define the category $\ssMod{\LL_Y}{\eta}$.

\begin{Theorem}\label{nearby-functor}
  Let $Y$ be an irreducible component $D$.  Then $\Phi _Y: \Hig (X,
  D)\to \Mod{\LL _Y^0}$ induces the functor $$\Phi_Y^{\mu}: \HIG{\mu}
  (X,D)\to \ssMod{\LL_Y^0}{\mu_Y},$$ where $\mu_Y$ is the image of
  $\mu$ under the restriction map $H^{2}_{\et} (X, \QQ_l)\to
  H^{2}_{\et} (Y, \QQ_l)$.
\end{Theorem}

\begin{proof}
  Let $(E, \theta)$ be an object of $\HIG{\mu} (X,D)$ and let us first
  assume that $\theta$ is nilpotent.
  Let $(V, \nabla)= C_{(\tilde X,\tilde D)}^{-1} (E, \theta)$. Let us
  denote by $S^{\bullet}$ (decreasing) Simpson's filtration on $(V,
  \nabla)$ and let $(E_1=\Gr_S(V),\theta _1)$ be the associated system
  of Hodge sheaves.

Let $E'$  be an $\LL _Y^0$-submodule of the $\LL_Y^0$-module $(E_Y, \theta |_Y)$.
Then by Lemma \ref{compatibility}
$$V'=s^*  C_{(\tilde Z,  \tilde D^Z)}^{-1} (E')\subset s^*C_{(\tilde Z,  \tilde D^Z)}^{-1} (\pi^*E _Y, \theta ')=
s^*(\pi ^* V_Y, \nabla ')=(V_Y, \nabla |_Y),$$ 
i.e., $V'$ is an $\LL _Y$-submodule of $(V_Y, \nabla |_Y)$.

By Theorem \ref{log-freeness} $E_1$ is locally free, so all $S^j$ are
locally free.  Thus we get an induced filtration $S^{\bullet}_Y$ of
$V_Y$ and
$$\Gr_{S_Y}^j (V_Y)=(\Gr_S^j V)_Y.$$
This filtration induces on $V'$ a filtration that we denote by abuse
of notation also by $S^{\bullet}_Y$. In this way we get an $\LL
_Y^0$-submodule $E_1'=\Gr_{S_Y} (V') \subset \Gr_{S_Y} (V_Y)=((E_1)_Y,
\theta_1|_Y)$. By construction we have $\mu_{H_Y}(E')=p \mu_{H_Y}(E_Y)$.

Now let us consider the canonical Higgs-de Rham sequence starting with $(E_0,\theta_0)=(E, \theta)$
(see Theorem \ref{Higgs-de-Rham})
$$ \xymatrix{
  & (V_0, \nabla _0)\ar[rd]^{\Gr _{S}}&& (V_1, \nabla _1)\ar[rd]^{\Gr _{S}}&\\
  (E_0, \theta _0)\ar[ru]^{C^{-1}}&&(E_1, \theta_1)\ar[ru]^{C^{-1}}&&...\\
}$$ Since $(E, \theta)$ is an object of $\HIG{\mu} (X,D)$, Theorem
\ref{log-freeness} implies that $(E_m, \theta _m)$ is an object of 
$\HIG{p^m \mu} (X,D)$ for all $m\ge 0$.  So we can apply the above described procedure at
all levels of the Higgs--de Rham sequence. This allows us to construct
a sequence $\{E'_m\} _{m\ge 0}$ of $\LL_Y^0$-modules such that
$E'_m\subset ((E_m)_Y, \theta_m|_Y)$ and $\mu_{H_Y}(E'_m)=p^m
\mu_{H_Y}(E')$.

Now we write $p^m=rs_m+q_m$ for some non-negative integers $s_m$ and
$0\le q_m<r$. Let us set $(G_m, \theta _{G_m}):=(E_m, \theta
_m)\otimes \det E^{-s_m}$. As in proof of Lemma \ref{stable-Higgs-de-Rham-sequence} we see
that the family of locally free slope $H$-semistable logarithmic Higgs
sheaves $\{(G_m, \theta _{G_m})\} _{m\ge 0}$ is bounded. This implies
that the family of sheaves $\{E_m\otimes \det E^{-s_m}\} _{m\ge 0}$ is
bounded and hence the family of their restrictions to $Y$ is
bounded. Therefore the numbers
$$\mu_{H_Y}(E'_m\otimes \det E_Y^{-s_m})=p^m \mu_{H_Y}(E')-rs_m \mu _{H_Y}(E_Y)= p^m (\mu_{H_Y}(E')-\mu_{H_Y} (E_Y))+q_m \mu _{H_Y} (E_Y)$$
are uniformly bounded from the above. Hence we get $\mu_{H_Y}(E')\le
\mu_{H_Y} (E_Y)$, i.e., the $\LL_Y^0$-module $(E_Y, \theta |_Y)$ is
slope $H_Y$-semistable.

\medskip

Now let us consider the general case.  Let $(E, \theta)$ be an object
of $\HIG{\mu} (X,D)$.  By Theorem \ref{deformation-to-system} there exists 
a decreasing Griffiths transverse filtration  $N^{\bullet }$ of $E$ such that the 
associated graded  $(E_0,\theta_0):= \Gr _N (E, \theta)$  is a slope $H$-semistable 
system of logarithmic Hodge sheaves (in particular, $\theta_0$ is nilpotent). 
Moreover, by Theorem \ref{log-freeness}
$E_0$ is locally free. By the first part of the proof we know that the
$\LL^0_Y$-module $((E_0)_Y, \theta_0|_Y)$ is semistable. Then by
openness of semistability $(E _Y, \theta|_Y )$ is also a semistable
$\LL_Y^0$-module.
\end{proof}

Let $\nMod{\LL_Y^0}{\eta}$ the the full subcategory of
$\ssMod{\LL_Y^0}{\eta}$, whose objects are $\LL_Y^0$-modules $E$ with
nilpotent $\Res _E$.  Replacing $\LL_Y^0$ by $\LL_Y$ we get the
definition of $\nMod{\LL _Y}{\eta}$.

\begin{Theorem} \label{passing-from-L-to-Hig}
  Let us fix a class $\eta\in H^{2}_{\et} (Y, \QQ_l)$ for some $l\ne
  p$.  The functor $\Upsilon ^0: \nilMod{\LL_Y^0}\to \Hig (Y,D^Y)$
  from Lemma \ref{functor-Upsilon} induces the functor $$\Upsilon
  ^0_{\eta}: \nMod{\LL_Y^0}{\eta}\to \HIG{\eta} (Y,D^Y).$$  In particular,
  for any object $E$ of $\nMod{\LL_Y^0}{\eta}$ we have for all $m\ge 1$
$$c_m(E)=\binom{r}{m}\eta ^m$$
in $H^{2m}_{\et}(Y, \QQ_l)$.
\end{Theorem}

\begin{proof}
  Let $E$ be an object of $\nMod{\LL_Y^0}{\eta}$. We need to prove that
  every quotient $(\Gr _j^ME, \theta_j)$ of the monodromy filtration
  $M_{\bullet}$ of $E$ is locally free, slope $H_Y$-semistable with
  $c_m(\Gr _j^ME)=\binom{r_j}{m}\eta ^m$ for all $m\ge 1$, where
  $r_j=\rk \Gr_j^ME$. This also implies that
  $c_m(E)=\binom{r}{m}\eta ^m$ for all $m\ge 1$.

  By Lemma \ref{semistable-quotients} we know that every quotient $\Gr
  _j^ME$ of the monodromy filtration $M_{\bullet}$ of $E$ is slope
  $H_Y$-semistable (as an $\LL_Y^0$-module) with $\mu _{H_Y}(\Gr
  _j^ME)=\mu _{H_Y}(E)$. We also know that $\Gr _j^ME$ is endowed with a
  natural logarithmic Higgs field $\theta_j^M$ on $(Y, D^Y)$, coming
  from the $\LL _Y^0$-action and triviality of the residue of $\Gr
  _j^ME$. Since any logarithmic Higgs subsheaf of $(\Gr _j^ME,
  \theta_j)$ has a canonical structure of an $\LL_Y^0$-submodule, the
  pair $(\Gr _j^ME, \theta_j)$ is slope $H_Y$-semistable.  Therefore by
  Theorem \ref{strong-log-freeness} all quotients $\Gr _j^ME$ are
  locally free with $c_m(\Gr _j^ME)=\binom{r_j}{m}\eta ^m$.  
\end{proof}

\medskip

\begin{Corollary} \label{ss-filtration-of-restriction}
Any element in the essential image of the functor 
$$\Phi_Y^{0}: \HIG{\mu} (X,D)\to \ssMod{\LL_Y^0}{\mu _Y}.$$
has a filtration whose quotients are elements of $\HIG{\mu _Y} (Y,D^Y)$.
\end{Corollary}

\begin{proof}
  Assume that an object $M$ of $\ssMod{\LL_Y^0}{\mu _Y}$ is isomorphic
  to $\Phi_Y^{0}(E, \theta)$ for some $(E,\theta)$ in $\HIG{\mu}
  (X,D)$.  In the last part of the proof of Theorem
  \ref{nearby-functor} we showed that there exists a decreasing
  Griffiths transverse filtration $N^{\bullet }$ of $(E,\theta)$ such
  that $\Phi_Y^{0} (\Gr _N (E, \theta))$ is an object of
  $\nMod{\LL_Y^0}{\eta}$.  In particular, by Theorem
  \ref{passing-from-L-to-Hig} every quotient in the monodromy
  filtration of $\Phi_Y^{0} (\Gr _N (E, \theta))$ is an element of
  $\HIG{\mu _Y} (Y,D^Y)$.  The proof finishes by remarking that
  $N^{\bullet}$ induces an analogous filtration on $M$.
\end{proof}

\section{Semistability and semipositivity}

In this section we prove Theorem \ref{semipositivity} and show some of its applications mentioned in the introduction.

\subsection{General results on semistability}

Unless otherwise stated, in this subsection $(X, D)$ stands for a smooth log pair defined over 
an algebraically closed field $k$of positive characteristic. We assume that the pair $(X, D)$ is liftable to $W_2(k)$
and we fix its lifting $(\tilde X, \tilde D)$.

 Let $C$ be a smooth projective curve and let $\nu: C\to X$ be a 
separable morphism. Let $D'$ be the sum of irreducible components 
of $D$ that do not contain $\nu (C)$ and let $D'_C=(\nu^{-1}(D'))_{\reduced}$.

\begin{Definition}\label{strongly-liftable}
We say that  $\nu : C\to (X,D)$ is \emph{strongly liftable to $W_2(k)$}, if there exists 
a good lifting $\tilde \nu : (\tilde C, \tilde D'_C) \to (\tilde X, \tilde D')$ (see
Definition \ref{good-lifting}) of $\nu: (C, D_C')\to (X, D')$ such that for every irreducible 
component $Y$ of $D$ containing $C$, $\tilde \nu$ factors through $\tilde C\to \tilde Y$.
\end{Definition}

In the above definition we write $\nu : C\to (X,D)$ to keep in mind
that being strongly liftable to $W_2(k)$ depends not only on $\nu :
C\to X$ but also on the choice of the normal crossing divisor $D$ (in fact, it also  depends on the choice of lifting 
 $(\tilde X. \tilde D)$ of $(X, D)$).

\begin{Theorem} \label{semipositivity-mu} Let $(E, \theta)$ be an
  object of $\HIG{\mu} (X,D)$.  Let $C$ be a smooth projective curve
  and let $\nu : C\to (X,D)$ be a morphism that is strongly
  liftable to $W_2(k)$. Then the induced $\Sym ^{\bullet} \nu^*
  T_X(\log \, D)$-module $\nu ^*E$ is semistable.  In particular, if
  $G$ is a subsheaf of the kernel of $\nu^*\theta : \nu^*E\to \nu^*
  E\otimes \nu^*\Omega_X(\log\, D)$ then $\mu(G)\le \mu (\nu^*E)$.
\end{Theorem}

\begin{proof}
  The proof is by induction on the dimension of $X$. In dimension
  $n=1$ the required assertion follows from Theorem
  \ref{curve-restriction}, so let us assume that $n\ge 2$.  As in the
  proof of Theorem \ref{nearby-functor} there exists a decreasing
  Griffiths transverse filtration $N^{\bullet }$ of $E$ such that the
  associated graded $(E_0,\theta_0):= \Gr _N (E, \theta)$ is an object
  of $\HIG{\mu} (X,D)$ with nilpotent $\theta_0$. Since $\Gr _{\nu ^* N}
  (\nu^*E, \nu^*\theta)= (\nu^*E_0, \nu^*\theta_0)$, by openness of
  semistability, if the $\Sym ^{\bullet} \nu^* T_X(\log \, D)$-module
  $\nu ^*E_0$ is semistable then the $\Sym ^{\bullet} \nu^* T_X(\log
  \, D)$-module $\nu ^*E$ is semistable.  So in the following we can
  assume that $\theta$ is nilpotent.

  If $\nu (C)$ is not contained in $D$ then   $\nu^*(E, \theta)$ is
  semistable by Theorem \ref{curve-restriction} (for this we do not
  need nilpotence of $\theta$). Since any $\Sym ^{\bullet} \nu^*
  T_X(\log \, D)$-submodule of $\nu^*E$ defines a Higgs subsheaf of
  $\nu^*(E, \theta)$, this implies that $\nu ^*E$ is semistable as a
  $\Sym ^{\bullet} \nu^* T_X(\log \, D)$-module.

  If $\nu (C)$ is contained in $D$ then we choose an irreducible component $Y$ 
  of $D$ containing $\nu (C)$ and as before we set $D^Y =(D-Y)|_Y$. By definition 
  of strong liftability, the morphism $C\to (Y, D^Y)$ is also strongly liftable to $W_2 (k)$.
  By Theorem
  \ref{nearby-functor} $E':=\Phi^0_Y (E, \theta)$ is an element of
  $\nMod{\LL_Y^0}{\mu_Y}$.  By Theorem \ref{passing-from-L-to-Hig} we
  know that $E'$ has a filtration $M_{\bullet}$ whose associated
  graded $E''=\Gr^M(E')$ is an element of $\HIG{\mu _Y} (Y,D^Y)$.
  Hence by the induction assumption the induced $\Sym ^{\bullet} \nu^*
  T_Y(\log \, D^Y)$-module $\nu ^*E''$ is semistable.  Equivalently,
  $\nu ^*E''$ is semistable as a $\nu^*\LL_Y^0$-module.
  
But $\nu^*M_{\bullet}$ is a filtration of $\nu^* E'$ by $\nu^*\LL_Y^0$-submodules and the 
associated graded is equal to  $\nu ^*E''$  (here we use the fact that $E''$ is locally free).
So by openness of semistability $\nu^* E'$ is semistable as a $\nu^*\LL_Y^0$-module.
This is equivalent to saying that  $\nu ^*E$ is semistable as a $\Sym ^{\bullet}
  \nu^* T_X(\log \, D)$-module, which finishes the induction step.
  
 The last part of the theorem follows from the fact that $\ker \nu^*\theta$ with trivial action is a  $\Sym ^{\bullet}
  \nu^* T_X(\log \, D)$-submodule of $\nu ^*E$.
\end{proof}

\begin{Corollary} \label{Bru-p}
 Let $(E, \theta)$ be an object of $\HIG{0} (X,D)$. 
If $E'$ is a locally split subsheaf of $E$ contained in the kernel of $\theta$
then its dual $(E')^*$ is $W_2$-nef.
\end{Corollary}

\begin{proof}
If $E'$ is a locally split subsheaf of $E$ then for any  smooth projective curve $C$ and any 
morphism  $\nu : C\to  X$,  $\nu^*E'$ is a subsheaf of $\nu^*E$. Moreover, the image of 
$\nu^*(\ker \theta)$ in $\nu^*E$ is contained in $\ker \nu^*\theta$, so $\nu^*E'\subset \ker \nu^*\theta$. 
So if $\nu$ is separable and liftable to $W_2(k)$, then by the above theorem any subsheaf of $\nu^*E'$ has a nonpositive degree.
Passing to the dual of $\nu^*E'$, we get the required assertion.
\end{proof}

A standard spreading out arguments show that Theorem
\ref{semipositivity-mu} implies the following result:

\begin{Theorem} \label{semipositivity-0} 
  Let $(E, \theta)$ be a locally free logarithmic Higgs sheaf on a
  smooth log pair $(X,D)$ defined over an algebraically closed field
  of characteristic zero. Assume that it has vanishing Chern classes
  in $H^{2*}(X, \QQ)$ and it is slope semistable with respect to some
  ample polarization.  Let $\nu : C\to X$ be any morphism from some
  smooth projective curve.  Then the induced $\Sym ^{\bullet} \nu^*
  T_X(\log \, D)$-module $\nu ^*E$ is semistable.  In particular, if
  $G$ is a subsheaf of the kernel of $\nu^*\theta : \nu^*E\to \nu^*
  E\otimes \nu^*\Omega_X(\log\, D)$ then $\deg G\le 0$.
\end{Theorem}

\begin{Remark}
\begin{enumerate}
\item For the first part of Theorem \ref{semipositivity-0} one can replace the assumption
  that $E$ has vanishing Chern classes with assumption that
  $r^ic_m(E)=\binom{r}{m}(c_1(E))^m$ for all $m\ge 2$ in $H^{2*}(X,
  \QQ)$.
\item In Theorem \ref{semipositivity-0} the assertion holds if we
  replace curve $C$ by any smooth polarized variety. This immediately
  follows from the fact that semistability on a general complete
  intersection curve implies semistability on the original variety.
\item  A posteriori one can see that it is possible to obtain proof of the
  above theorem without passing to positive characteristic. In case
  $(E,\theta)$ comes from a real graded-polarized family of mixed
  Hodge structures it is possible to use Mochizuki's version of
  Simpson's correspondence to adapt Brunebarbe's proof \cite[Theorem
  4.5]{Br1} to obtain the above theorem. This strategy can be also
  generalized to deal with arbitrary systems of logarithmic Hodge
  bundles. The general case needs a logarithmic version of
  \cite[Theorem 2]{Si} (cf. Theorem \ref{log-freeness}), which again
  can be obtained using Mochizuki's results. Passing to non-zero $\mu$
  as in Theorem \ref{semipositivity-mu} can de done using
  Theorem \ref{0-mu}.
\end{enumerate}
\end{Remark}

Theorem \ref{semipositivity-0} implies the following result
generalizing \cite[Theorem 1.2]{Br2} from polystable to the semistable
case:

\begin{Corollary} \label{semipositivity-cor-0} Let $(E, \theta)$ be a
  locally free logarithmic Higgs sheaf on a smooth log pair $(X,D)$
  defined over an algebraically closed field of characteristic
  zero. Assume that it has vanishing Chern classes in $H^{2*}(X, \QQ)$
  and it is slope semistable with respect to some ample
  polarization. If $E'$ is a locally split subsheaf of $E$ contained
  in the kernel of $\theta$ then its dual $(E')^*$ is nef.
\end{Corollary}

\subsection{Geometric applications}

In this subsection we give several geometric applications of Corollary \ref{Bru-p} in more or 
less increasing degree of generality showing how to adjust some arguments.
We fix the following notation.  Let $X$ and $Y$ be smooth projective varieties defined over an algebraically closed field $k$ of characteristic $p$ and let $f:X\to Y$ be a  surjective $k$-morphism of relative dimension $d$.
Moreover, $i$ and $j$ are arbitrary non-negative integers.

\begin{Corollary}\label{canonical-smooth}
Assume that $f$ is smooth $d<p $ and there exists  a lifting 
  $\tilde f: \tilde X\to \tilde Y$ of $f$ to
  $W_2(k)$. Then $(R^if_*^{dR}\cO_X, \nabla _{GM})$, where $\nabla
  _{GM}$ is the Gauss-Manin connection, is a locally free semistable
  sheaf with an integrable connection and vanishing Chern classes. In
  particular, $R^jf_*\omega_{X/Y}$ is a $W_2$-nef locally free sheaf on
  $Y$.
\end{Corollary}

\begin{proof} 
  By \cite[Theorem 4.17]{OV} we have a canonical isomorphism
$$C^{-1}_{\tilde Y} (\Gr _FR^if_*^{dR}\cO_X, \kappa) \simeq (R^if_*^{dR}\cO_X, \nabla _{GM}),$$
where $F^{\bullet}$ is the Hodge filtration and $\kappa$ is the
associated graded (i.e., the cup-product with the Kodaira-Spencer
mapping). If $Y$ is projective then the above isomorphism implies that
both $(R^if_*^{dR}\cO_X, \nabla _{GM})$ and $(\Gr _FR^if_*^{dR}\cO_X,
\kappa)$ are semistable as we have a periodic Higgs-de Rham sequence
of $(\Gr _FR^if_*^{dR}\cO_X, \kappa)$ (here we use \cite[Proposition
1]{La2}). So Corollary \ref{semipositivity-cor} implies that the first
non-zero piece of the Hodge filtration of $R^if_*^{dR}\cO_X$, i.e.,
$R^{i-d}f_*\omega_{X/Y}$, is a $W_2$-nef locally free sheaf on $Y$.
\end{proof}

\begin{Remark}
  In the complex case the above corollary is precisely the result of
  Griffiths (see \cite[Corollary 7.8]{Gr}), who showed that if $f:
  X\to Y$ is a smooth morphism of smooth projective varieties, then
  the direct image $f_*\omega_{X/Y}$ of the relative canonical bundle
  is locally free and nef.
\end{Remark}

\begin{Corollary}\label{canonical-Katz}
Let $D$ be
  a divisor on $X$ which is a union of divisors, each of which is
  smooth over $Y$, and which have normal crossings relative to
  $Y$. Let us assume that $f$ us smooth, $d<p$ and there exists a lifting $\tilde f: \tilde
  X\to \tilde Y$ of $f$ to $W_2(k)$ and a compatible lifting $\tilde
  D$ of $D$. Then $ (R^if_*\Omega_{X/Y} ^{\bullet}(\log \, D), \nabla
  _{GM})$ is semistable with vanishing Chern classes.  In particular,
  $R^jf_*\omega_{X/Y}(D)$ is a $W_2$-nef locally free sheaf on $Y$.
\end{Corollary}

\begin{proof} 
  The proof is the same as that of Corollary \ref{canonical-smooth}
  except that we need to reformulate Katz's \cite[Theorem 3.2]{Katz}
  using the inverse Cartier transform (cf. \cite[Example 3.17 and
  Remark 3.19]{OV}). In this way we get a canonical isomorphism
$$C^{-1}_{\tilde Y} (\Gr _FR^if_*\Omega_{X/Y} ^{\bullet}(\log \, D), \kappa) 
\simeq (R^if_*\Omega_{X/Y} ^{\bullet}(\log \, D), \nabla _{GM}).$$
\end{proof}

\medskip

One can also get similar theorems as above in the case of ``unipotent local
monodromies'', e.g., for semistable reductions. 
Before stating the corresponding result let us recall the definition of a semi-stable reduction (see \cite[Definition 1.1]{Il}).
Let $S$ be a scheme and let $X$ and $Y$ be smooth $S$-schemes, $f: X\to Y$ an $S$-morphism and $B\subset Y$ a normal crossing divisor relative to $S$, $D:=X\times _Y B$.  
We say that that $f: X\to Y$ is \emph{semi-stable} (or  $f$ has \emph{a semi-stable reduction along $B$})
if  locally in the \'etale topology on $X$, $f$ is a product of $S$-morphisms of the following type: 
\begin{enumerate}
	\item the projection $\pi _1: \AA ^n_S\to \AA^1_S$, $B=0$, 
	\item $h: \AA^n_S =\Spec \cO_S [x_1,...,x_n] \to \AA^1_S= \Spec \cO_S [y]$, $h^*y=x_1...x_n$, $B=V(y)$.
\end{enumerate}

\begin{Corollary}\label{canonical-Illusie}
  Let $B$ be a normal crossing divisor on $Y$ and assume that $f$ has a semi-stable
  reduction along $B$.  Let us set $D=f^{-1}(B)$.  Assume that
  there exists a lifting $\tilde f: (\tilde X, \tilde D)\to (\tilde
  Y, \tilde B)$ of $f$ to $W_2(k)$ with $\tilde f$ a semi-stable
  reduction along $\tilde B$. Assume that $p>d+\dim Y$.
  Then $$(R^if_*\Omega_{X/Y} ^{\bullet}(\log \, D/B), \nabla
  _{GM})$$ is a semistable locally free $\cO_Y$-module with an
  integrable logarithmic connection on $(Y,B)$. In particular,
  $R^jf_*\omega_{X/Y}(D)$ is a $W_2$-nef locally free sheaf on $(Y, B)$.
\end{Corollary}

\begin{proof}
  Again the proof is the same as that of Corollary
  \ref{canonical-smooth}, except that now one needs to use
  \cite[Theorem 4.7]{Il} and check that the corresponding result describes an isomorphism
$$C^{-1}_{(\tilde Y, \tilde B)} (\Gr _FR^if_*\Omega_{X/Y} ^{\bullet}(\log \, D/B), \kappa) 
\simeq (R^if_*\Omega_{X/Y} ^{\bullet}(\log \, D/B), \nabla
_{GM}).$$ Assumptions of this theorem are satisfied due to
\cite[Corollary 2.4]{Il} and our assumption $p>d+\dim Y$.  We leave
checking cumbersome details to the interested reader.
\end{proof}

In characteristic zero, the above result is almost the same as
\cite[Theorem 5]{Ka}.

\medskip

\begin{Remark}
  One can also combine Corollaries \ref{canonical-Katz} and
  \ref{canonical-Illusie} using \cite[4.22]{Il}). It is also
  possible to further generalize these results and deal with
  push-forwards of Fontaine modules as in \cite[Theorem 4.17]{OV}
and the corresponding log versions.
\end{Remark}

\section{Appendix: functoriality of the inverse Cartier transform}

In this appendix we prove the functoriality of the inverse Cartier
transform. In the non-logarithmic case functoriality follows from
\cite[Theorem 3.22]{OV}. Unfortunately, although it seems very likely
that an analogue of this result holds in the logarithmic case, this
part of their paper was never generalized.

In the following instead of dealing with a general theory that would
demand a lot of space and additional notation, we deal only with the
simple cases used in the paper. Instead of using the general framework
of \cite{Sc} that follows \cite{OV}, we use an explicit description of
the Ogus--Vologodsky correspondence provided in \cite{LSZ2} and
\cite[Appendix]{LSYZ}.

	Let $k$ be an algebraically closed field of positive characteristic
and let $f: (Y, B)\to (X,D)$ be a $k$-morphism of smooth log pairs over $k$.

\begin{Definition} \label{good-lifting} 
  We say
  that $f$ has a \emph{good lifting to $W_2(k)$} if $f$ lifts to a
  morphism of smooth log pairs $\tilde f: (\tilde Y, \tilde B)\to
  (\tilde X, \tilde D)$ over $W_2(k)$ such that locally in the \'etale topology on $\tilde X$, $\tilde f$ admits compatible liftings of the Frobenius morphisms, 
  i.e., we can cover $\tilde X$ with images of \'etale $W_2(k)$-morphisms $\tilde U\to \tilde X$ and
  $\tilde Y$ with images of \'etale $W_2(k)$-morphisms $\tilde V\to \tilde Y$ so that
\begin{enumerate}
\item there exists  $\tilde F_{U}: \tilde U\to \tilde U$ lifting the
  Frobenius morphism $F_U$, where $U=\tilde U\otimes _{W_2(k)} k$, so that $\tilde F_{U}^{-1}(\tilde D)=p \tilde D,$
\item there exists  $\tilde F_{V}: \tilde V\to \tilde V$ lifting the
  Frobenius morphism $F_V$, where $V=\tilde V\otimes _{W_2(k)} k$,  so that $\tilde F_{U}^{-1}(\tilde B)=p \tilde B,$
\item there exists $\tilde f _V: \tilde V \to \tilde U$ lifting $\tilde f$ such  that  the diagram
$$\xymatrix{
\tilde V\ar[r]^{\tilde f _V}\ar[d]^{\tilde F_V}&\tilde U\ar[d]^{\tilde F_U}\\
\tilde V\ar[r]^{\tilde f _V}&\tilde U\\
}$$
is commutative.
\end{enumerate}
In this case we say that $\tilde f$ is a \emph{good lifting of $f$ to $W_2(k)$}. 
\end{Definition}

Clearly, if $\tilde f$ is an open embedding then it is a good lifting. Similarly, a composition of good liftings is a good lifting. It is also easy to see that the standard Frobenius morphism given by raising elements to their $p$-th power gives the following proposition:

\begin{Proposition}  
Assume $f$ lifts to a morphism of smooth log pairs $\tilde f: (\tilde Y, \tilde B)\to
  (\tilde X, \tilde D)$ over $S=\Spec W_2(k)$ such that locally in the \'etale topology on 
  $\tilde X$, $\tilde f$ is a composition of products of $S$-morphims of the following type:
  \begin{enumerate}
  	\item the projection $\pi _1: \AA ^n_S\to \AA^1_S$, $\tilde B=0$,  $\tilde D=0$, 
  	\item the embedding $i _1: \AA ^1_S\to \AA^n_S$, $\tilde B=0$,  $\tilde D=0$, 
  	\item $h:  \AA^m_S =\Spec \cO_S [y_1,...,y_m] \to \AA^n_S =\Spec \cO_S [x_1,...,x_n]$, 
  $\tilde B= V(\prod_{i=1}^{m} y_i )$, $\tilde D=V(\prod_{j=1}^{n} x_j)$ and 
  for $j=1,...n$ we have
  $$h^*(x_j)=\prod_{i=1}^{m} y_i^{a_{ij}},$$
where $a_{ij}$ are some non-negative integers. 
\end{enumerate}
Then $\tilde f$ is a {good lifting of $f$ to $W_2(k)$}. 
\end{Proposition}

\begin{Remark} 
  It is easy to see that any log-smooth lifting $\tilde f$ of $f$ to
  $W_2(k)$ is a good lifting. 
  One can also see that almost every reduction of a morphism of
  smooth log pairs from characteristic zero to positive characteristic
  gives rise to a good lifting. For example, in the case $\tilde B=0$ and $\tilde D=0$ one can decompose any morphism
 of smooth schemes over an algebraically closed field into a composition of a closed embedding  
 and a smooth morphism. A smooth closed subvariety of a smooth variety is locally in the \'etale topology a product of maps 
 of type 2 and a smooth morphism is locally in the \'etale topology a product of maps of type 1.
\end{Remark}

\medskip

Let $\Hig _{\le p-1}^{\rm lf}(X,D)$ be the full subcategory of $\Hig
(X,D)$ consisting of locally free logarithmic Higgs sheaves with
nilpotent Higgs field of level less or equal to $p-1$.  Let $\Mic
_{\le p-1}(X,D)$ be the full subcategory of $\Mic (X,D)$ consisting of
$\cO_X$-modules with an integrable logarithmic connection whose
logarithmic $p$-curvature is nilpotent of level less or equal to $p-1$
and the residues are nilpotent of order less than or equal to $p$.

\begin{Theorem} \label{functoriality-Cartier}\footnote{After sending the preprint, the author was informed 
by K. Zuo that together with R. Sun and J. Yang 
they checked compatibility of the inverse Cartier transform for double covers of $\PP^1$.}
  Let $f: (Y, B)\to (X,D)$ be a morphism of smooth log pairs that has
  a good lifting $\tilde f: (\tilde Y, \tilde B)\to (\tilde X, \tilde
  D)$ to $W_2(k)$. Then we have an isomorphism of functors
$$f^*\circ C^{-1}_{ (\tilde X, \tilde D)} \simeq C^{-1}_{ (\tilde Y, \tilde B)}  \circ f^* : {\Hig}_{\le p-1}^{\rm lf}(X,D)
\to {\Mic} _{\le p-1}(Y,B).$$
\end{Theorem}

\begin{proof}
 
\emph{Step 1.} Let us first assume that there exist global compatible logarithmic liftings of
  the Frobenius morphism on $X$ and $Y$, i.e., 
\begin{enumerate}
\item there exists  $\tilde F_{X}: \tilde X\to \tilde X$ lifting the
  Frobenius morphism $F_X$ so that
$$\tilde F_{X}^* \cO_{\tilde X}(-\tilde D)=\cO_{\tilde X}(-p \tilde D),$$
\item there exists  $\tilde F_{Y}: \tilde Y\to \tilde Y$ lifting the
  Frobenius morphism $F_Y$ so that
$$\tilde F_{Y}^* \cO_{\tilde Y}(-\tilde B)=\cO_{\tilde Y}(-p \tilde B),$$
\item the diagram
$$\xymatrix{
\tilde Y\ar[r]^{\tilde f}\ar[d]^{\tilde F_Y}&\tilde X\ar[d]^{\tilde F_X}\\
\tilde Y\ar[r]^{\tilde f}&\tilde X\\
}$$
is commutative.
\end{enumerate}

The first condition implies that there exists a uniquely defined $\zeta _X$ such that the diagram 
$$ \xymatrix{
  \tilde F_X^*\Omega_{\tilde X}^1(\log \, \tilde D)\ar@{>>}[d]
  \ar[r]^{d \tilde F_{X}} &
  \Omega_{\tilde X}^1(\log \, \tilde D)\\
  F_X^* \Omega_{X}^1(\log \, D)\ar[r]^{\zeta_X}&
  \Omega_{X}^1(\log \, D)\ar[u]_{p}\\
}$$ is commutative. The second condition gives $\zeta_Y$ with a
similar diagram for $(\tilde Y, \tilde B)$. The third condition shows
that we have a commutative diagram
$$ \xymatrix{
 \tilde F_Y^*\tilde f^*\Omega_{\tilde X}^1(\log \, \tilde   D)\ar[dr]_{\tilde F_Y^*(d\tilde f)} \ar@{=}[r]& \tilde f^*\tilde F_X^*\Omega_{\tilde X}^1(\log \, \tilde
    D) \ar[rr]^{\tilde f^*(d \tilde  F_{X})} && \tilde f^*\Omega_{\tilde X}^1(\log \, \tilde D)\ar[d]_{d\tilde f}\\
& \tilde F_Y^* \Omega_{\tilde Y}^1(\log \, \tilde B)\ar[rr]^{d\tilde F_Y}&&
  \Omega_{\tilde Y}^1(\log \, \tilde B).\\
}$$
Together with the previous two diagrams this shows that the  diagram
$$ \xymatrix{
  F_Y^*f^*\Omega_{X}^1(\log \, D)\ar[dr]_{F_Y^*(df)} \ar@{=}[r]&
  f^*F_X^*\Omega_{ X}^1(\log \,
  D) \ar[rr]^{f^*(\zeta_{X})} && f^*\Omega_{X}^1(\log \, D)\ar[d]_{d f}\\
  & F_Y^* \Omega_{Y}^1(\log \, B)\ar[rr]^{\zeta_Y}&&
  \Omega_{ Y}^1(\log \, B)\\
}$$ 
is also commutative.
Now let $(E, \theta)$ be an object of $\Hig _{\le p-1}^{\rm lf}(X,D)$
and let us write $f^*(E, \theta)=(f^*E, \theta _Y)$. Then we set
$C^{-1}_{(\tilde X, \tilde D)}(E,\theta)=(F_X^*E, \nabla)$, where
$$\nabla := \nabla_{can}+({\id} _{F_X^*E}\otimes \zeta_{X})\circ (F^*_{X}\theta)$$
and $ \nabla_{can}$ is the canonical connection on $F_X^*E$ appearing
in Cartier's descent theorem (i.e., $ \nabla_{can}$ is the differentiation
along the fibers of the Frobenius morphism).  Similarly, we can define
$C^{-1}_{ (\tilde Y, \tilde B)}$.  Since $f^*(F_X^*E, \nabla
_{can})=(F_Y^*(f^*E), \nabla _{can})$, the above diagram shows that
$$f^*C^{-1}_{(\tilde X, \tilde D)}(E,\theta)=f^*(F_X^*E, \nabla)=(F_Y^*f^*E, \nabla_{can}+
({\id} _{F_Y^*f^*E}\otimes \zeta_{Y})\circ (F^*_{Y}\theta _Y))=C^{-1}_{ (\tilde Y, \tilde B)} f^*
(E,\theta).$$

\medskip

\emph{Step 2.}
Now let us assume that we have two pairs 
$(\tilde F_{X}^1, \tilde F_{Y}^1)$ and $(\tilde F_{X}^2, \tilde F_{Y}^2)$
of compatible global logarithmic liftings 
of the Frobenius morphism on $X$ and $Y$.
There exist an $\cO_{X}$-linear map $h_{12}^X$
such that the following diagram is commutative
$$\xymatrix{
  \cO_{\tilde X}\ar[r]^{(\tilde F_X^2)^*- (\tilde F_X^1)^*}\ar@{>>}[d]& p\tilde F_*\cO_{\tilde X}\\
  \cO_X\ar[d]^d&\\
  \Omega^1_{X}(\log \, D)\ar[r]^{h_{12}^X}&F_*\cO_X\ar[uu]_{\simeq}^{p}\\
}$$ By abuse of notation we let $h_{12}^X: F^*\Omega _{X}(\log \, D)
\to \cO_{X}$ be adjoint to $h_{12}^X$.  Similarly, one can define
$h_{12}^Y:F^*\Omega _{Y}(\log \, B) \to \cO_{Y}$. It is
straightforward to check that we have a commutative diagram:
$$ \xymatrix{
  F_Y^*f^*\Omega_{X}^1(\log \, D)\ar[dr]_{F_Y^*(df)} \ar@{=}[r]&
  f^*F_X^*\Omega_{ X}^1(\log \,
  D) \ar[rr]^{f^*(h_{12}^X)} && f^*\cO_X\ar@{=}[d]\\
  & F_Y^* \Omega_{Y}^1(\log \, B)\ar[rr]^{h_{12}^Y}&&
  \cO_{ Y}.\\
}$$ 
Now let us define a map 
$$\tau_{12}^X: F^*E\stackrel{F^*\theta }{\longrightarrow} F^*E\otimes 
F^* \left(\Omega_{X} (\log \, D)\right) \stackrel{\id\otimes h_{12}^X}{\longrightarrow} 
F^*E.$$
Similarly we define $\tau_{12}^Y: F^*(f^*E)\to F^*(f^*E)$. The above diagram shows that $\tau_{12}^Y=f^*\tau_{12}^X$.

\medskip
\emph{Setp 3.}
Now we consider the general situation.  Let $(E, \theta)$ be an object
of $\Hig _{\le p-1}^{\rm lf}(X,D)$. By assumption there exist \'etale coverings
$\{\tilde U_{\alpha}\} _{\alpha\in I}$ of $\tilde X$ and $\{\tilde
V_{\alpha}\} _{\alpha\in I}$ of $\tilde Y$ such that we have
compatible logarithmic liftings $(\tilde F_{X,\alpha}, \tilde
F_{Y,\alpha})$ of the Frobenius morphisms $F_{X,\alpha}: U_{\alpha}\to
U_{\alpha}$ and $F_{Y,\alpha}: V_{\alpha}\to V_{\alpha}$.

Let us recall the construction of $(M, \nabla)=C^{-1}_{(\tilde X,
  \tilde D)}(E,\theta)\in \Mic (X,D)$ after \cite{LSZ2} and
\cite[Appendix]{LSYZ}.  Over each $U_{\alpha}$ we define $(M_{\alpha},
\nabla_{\alpha})$ by using Step 1 and setting 
$$(M_{\alpha}, \nabla_{\alpha}) := C^{-1}_{(\tilde U_{\alpha},  \tilde D\cap \tilde U_{\alpha})}(E,\theta).$$
Over $U_{\alpha\beta}=U_{\alpha}\times_X  U_{\beta}$ we can use
two liftings $\tilde F_{X,\alpha}|_{U_{\alpha\beta}}$ and $\tilde
F_{X,\beta}|_{U_{\alpha\beta}}$ of the Frobenius morphism $F:
U_{\alpha\beta} \to U_{\alpha\beta}$ to define $\tau ^X_{\alpha\beta}:
F^*(E_{U_{\alpha\beta}})\to F^*(E_{U_{\alpha\beta}})$ as in Step 2. Then we
glue $(M_{\alpha}, \nabla_{\alpha})$ and $(M_{\beta}, \nabla_{\beta})$
over $U_{\alpha\beta}$ to a global object $(M, \nabla)\in \Mic (X,D)$
using
$$g_{\alpha \beta}^X:=\exp (\tau_{\alpha\beta}^X)=
\sum_{i=0}^{p-1}\frac{(\tau_{\alpha\beta}^X)^i}{i!}.$$
Here we use the fact that the category of quasi-coherent sheaves in the Zariski and \'etale toposes of $X$ are equivalent 
(and we can replace a connection by an appropriate $\cO_X$-linear map using Grothendieck's description of connections).
We can also define  $\zeta_{Y,\alpha}$, $\tau^Y_{\alpha\beta}$ and $g^Y_{\alpha\beta}$. 
We already know that
$$f^*(M_{\alpha}, \nabla_{\alpha})=f^*C^{-1}_{(\tilde U_{\alpha},  \tilde D\cap \tilde U_{\alpha})}(E_{U_{\alpha}},\theta|_{U_{\alpha}})
=C^{-1}_{(\tilde V_{\alpha},  \tilde B\cap \tilde V_{\alpha})} \left( f^* (E,\theta)|_{V_{\alpha}}\right) $$ and
$\tau_{\alpha\beta}^Y=f^*\tau_{\alpha\beta}^X$. In particular,
$g_{\alpha\beta}^Y=f^*g_{\alpha\beta}^X$ which shows that gluing maps
agree and 
$$f^*C^{-1}_{(\tilde X, \tilde D)}(E,\theta)=C^{-1}_{ (\tilde Y, \tilde B)} f^*
(E,\theta).$$
\end{proof}

\begin{Remark}
  The above isomorphism of functors holds more generally without
  restricting to locally free logarithmic Higgs sheaves.  We added
  this assumption only to ensure that  $\Tor _1(f^*E,
  F_*\cO_{B_i})=0$ for all irreducible components $B_i$ of $B$. This
  allows us to conclude that the image is in $\Mic _{\le p-1}(Y,B)$.
\end{Remark}

\section*{Acknowledgements}

The author was partially supported by Polish National Centre (NCN)
contract numbers 2013/08/A/ST1/00804, 2015/17/B/ST1/02634 and 2018/29/B/ST1/01232.
The author would like to thank the referee for very careful reading of the paper, 
pointing out some gaps and for the numerous remarks that allowed to improve  the paper.

\end{document}